\documentclass[11pt,a4paper,reqno]{amsart}
\usepackage[includehead,includefoot,margin=25mm]{geometry}
\usepackage{todonotes}
\usepackage{times}
\usepackage{amsfonts} %
\usepackage{amsmath} %
\usepackage{amssymb} %
\usepackage{amsthm} %
\usepackage{enumerate} %
\usepackage{lipsum} %
\usepackage{float} %
\usepackage[ruled,vlined]{algorithm2e}
\usepackage{bbm} %
\usepackage{caption} %
\usepackage{subcaption} %
\usepackage{graphicx} %
\usepackage{nicefrac} 
\usepackage{color} %
\usepackage{hyperref}
\usepackage{tikz-cd}
\usepackage{algorithm2e}
\usepackage{enumitem}
\usepackage[toc,page]{appendix}
\usepackage{lipsum}
\hypersetup{
    colorlinks=true,       % false: boxed links; true: colored links
    linkcolor=blue,          % color of internal links
    citecolor=blue,        % color of links to bibliography
    filecolor=blue,      % color of file links
    urlcolor=blue           % color of external links
}
\usepackage[all]{xy}
\usepackage{imakeidx}
\makeindex

\newtheorem{theorem}{Theorem}[section] %
 
\newtheorem{corollary}[theorem]{Corollary} %
\newtheorem{lemma}[theorem]{Lemma} %
{\theoremstyle{remark} %
  \newtheorem{remark}[theorem]{Remark}} %
{\theoremstyle{definition} %
  \newtheorem{definition}[theorem]{Definition} %
  \newtheorem{example}[theorem]{Example} %
}
\newtheorem*{A}{Theorem A}
\newtheorem*{B}{Corollary B}

\newcommand{\PP}[0]{\ensuremath{\mathbb{P}}}
\newcommand{\CC}[0]{\ensuremath{\mathbb{C}}}
\newcommand{\ZZ}[0]{\ensuremath{\mathbb{Z}}}

\newcommand{\PGL}[0]{\ensuremath{\operatorname{PGL}}}

\newcommand{\GL}[0]{\ensuremath{\operatorname{GL}}}
\newcommand{\SL}[0]{\ensuremath{\operatorname{SL}}}

\usepackage[normalem]{ulem}
\usepackage{dynkin-diagrams}
\usepackage{multirow}

\RestyleAlgo{ruled}

\begin{document}

\title[]{Invariant Smooth Quartic Surfaces by all Finite Primitive Groups of $\operatorname{PGL}_4(\mathbb{C})$}

\author{Jose Avila}
\address{Departamento de Matem\'aticas, Universidad Del Valle, Cali, Colombia}  \email{jose.avila@correounivalle.edu.co}

\author{Guillermo Ortiz} %
\address{Departamento de Matem\'aticas, Universidad del Valle, Cali, Colombia} %
\email{guillermo.ortiz@correounivalle.edu.co}

\author{Sergio Troncoso}
\address{Departamento de Matem\'aticas, Universidad T\'ecnica
  Fe\-de\-ri\-co San\-ta Ma\-r\'\i a, Valpara\'\i
  so, Chile}  
\email{sergio.troncosoi@usm.cl}

\thanks{
  \mbox{\hspace{11pt}}{\it Key words}: Automorphism groups, Quartic Surface, Smooth Hypersurfaces}

\begin{abstract}For each finite primitive subgroup $G$ of $\PGL_4(\CC)$, we find all the smooth
$G$-invariant quartic surfaces. We also find all the faithful representations in $\PGL_4(\CC)$ of the smooth
quartic $G$-invariant surfaces by the groups:
$\mathfrak{A}_5,\mathfrak{S}_5, \operatorname{PSL_2(\mathbb{F}_7)},\mathfrak{A}_6,\mathbb{Z}_2^4\rtimes\mathbb{Z}_5$ and $\mathbb{Z}_2^4\rtimes D_{10}$. The primitive representation of these groups are precisely the subgroups of $\operatorname{PGL}_4(\mathbb{C})$ for which $\mathbb{P}^3$ is not $G$-super rigid.
As a byproduct, we show that the smooth quartic surface with the biggest group of projective automorphism  is given by  $\{ x_0^4 + x_1^4 + x_2^4 + x_3^4 + 12 x_0 x_1 x_2 x_3= 0\}$ (unique up to projective equivalence). 

%$\{ x_0^4 + x_1^4 + x_2^4 + x_3^4 + 12 x_0 x_1 x_2 x_3= 0\}$ is the most symmetric quartic smooth surface for which the group of projective automorphisms   is a primitive subgroup of $\PGL_4(\CC)$.
\end{abstract}

\maketitle

\setcounter{tocdepth}{1}

\tableofcontents

\section{Introduction}

It is well known that the simplest algebraic varieties are the hypersurfaces. However, it is a good place to test and find examples of general theories. In the history of algebraic geometry, there has been a lot of research around hypersurfaces, since the classical works of Cayley in the 19th century, to the more recent works about rigidity presented by J. Kollár \cite{kollar2019rigidity}, where the author synthesized a century of mathematical effort  which start at  times of Fano, and show that any  smooth hypersurfaces of dimension $n$ with degree $d=n+1$ are super rigid.
Concerning automorphisms of hypersurfaces, a remarkable
result due to Matsumura and Monsky, see \cite[Theo. 2]{matsumura1963automorphisms}, states that for hypersurfaces of degree $d$ and dimension $n$, except for  $n=2$ and $d=4$, the group of projective automorphisms of a hypersurface $X$, denoted by $\operatorname{PAut}(X)$, and the automorphisms of the hypersurface $X$, denoted by $\operatorname{Aut}(X)$, are the same group. More recently, V. González-Aguilera, A. Liendo, and P. Montero showed that for $d\geq 3$, $n\geq1$ and $(n,d)\neq (2,3),(3,4)$, the automorphism group of every smooth hypersurface of dimension $n-1$ and degree $d$ in $\mathbb{P}^n$ is liftable, (i.e., we can realize $\operatorname{Aut}(X)$ as a subgroup of $\operatorname{GL}(V)$) if and only if the number of variables and the degree are relative prime, i.e., $\gcd(n+1,d)=1$, see, e.g., \cite{gonzalez2020liftability}. On the other hand, we have that even for the most basic K3 surfaces $S$, the quartic surfaces, the groups $\operatorname{PAut}(S)$, and  $\operatorname{Aut}(S)$ are different, see, for instance,  \cite{matsumura1963automorphisms,gonzalez2020liftability}. Furthermore, to study the group $\operatorname{PAut}(S)$ it is necessary to go through the groups $G$ which are subgroups of $\operatorname{PGL}_4(\mathbb{C})$ and not just subgroups of $\operatorname{GL}_4(\mathbb{C})$. 

Recall that, a variety $X\subseteq \mathbb{P}^n$ is $G$-invariant for a subgroup $G<\operatorname{PGL_n(\mathbb{C})}$  if
%a conjugate of $G$ is contained in $\operatorname{Paut(X)}$.
$G \subset \operatorname{PAut}(X)$.
To find $G$-invariant smooth quartic surfaces in $\mathbb{P}^3$, we need to have a list of possible subgroups of $\operatorname{PGL_4(\mathbb{C})}$. The finite subgroups of $\operatorname{PGL_4(\mathbb{C})}$ were classified  by      H. Blichfeld  in his keystone book \cite[Chap. VII]{FiniteCollineationGroups}. In \cite[Appendix A]{Ivan2019}, the authors' present diagrams of all finite primitive groups of $\operatorname{PGL_4(\mathbb{C})}$ including their inclusions, see the Figures ~\ref{fig:1° - 12°},\ref{fig:13° - 21°}, \ref{fig:(A), (C) - (F), (G) y (K)}, and \ref{fig:(B) y (H)}. 

The main goal of our work is to find all the smooth quartic surfaces in $\mathbb{P}^3$ invariant by the finite primitive subgroups $G$ of $\operatorname{Aut(\mathbb{P}^3)}$. Also, we found all the smooth quartic surfaces invariant by the finite primitive subgroups of $\operatorname{PGL}_4(\mathbb{C})$ for which $\mathbb{P}^3$ is not $G$-super rigid. These groups are isomorphic to $\mathfrak{A}_5,\mathfrak{S}_5,\operatorname{ PSL_2(\mathbb{F}_7)},\mathfrak{A}_6, \mathbb{Z}_2^4\rtimes\mathbb{Z}_5$, and $\mathbb{Z}_2^4\rtimes D_{10}$, 
see \cite[Theo.1.3]{Ivan2019}.

Recall that $\mathbb{P}^3$ is said to be $G$-birational super-rigid, if:
\begin{enumerate}[noitemsep,topsep=0pt]
    \item There are no other $G$-Fano varieties $G$-equivariantly birational to $\mathbb{P}^3$ ,
    \item There are no $G$-equivariant birational map from $\mathbb{P}^3$ to a variety $X$ such that there is a
(non-birational) $G$-equivariant epimorphism $\pi :X \to Z$ where $\operatorname{dim}(X)> \operatorname{dim}(Z)\neq 0$
and whose general fiber is an irreducible rationally connected variety, and
\item  $\operatorname{Bir}^G(\mathbb{P}^3)=\operatorname{Aut}^G(\mathbb{P}^3)$, where $\operatorname{Bir}^G(\mathbb{P}^3)$ is the normalizer of the group $G$ in $\operatorname{Bir}(\mathbb{P}^3)$. Similarly, $\operatorname{Aut}^G(\mathbb{P}^3)$.
For a detailed explanation, see \cite[Def. 3.1.1]{cheltsov2015cremona}.
\end{enumerate}

The classification of groups of automorphisms of hypersurfaces is at an early state. Some results are known; for instance, in \cite{dolgachev2019automorphisms}, the authors   classify all possible automorphism groups of smooth cubic surfaces over an algebraically closed field of arbitrary characteristics.
For  non-singular cubic hypersurfaces in $\mathbb{P}^4$ over algebraically closed fields of characteristic $0$, the work is completely done, see \cite{wei2020automorphism}.  A partial classification of automorphisms  of cubic fourfolds is given in \cite{laza2022automorphisms}. Oguiso and Yu in \cite{oguiso2015automorphism} studied the automorphisms of the smooth quintic threefold.

In the case of quartic surfaces, there is no complete description. However, there are several works in that direction.
The quartic surfaces $G$-invariants by the alternating group $\mathfrak{A}_5$, icosahedron group, were found in \cite{faina2016s5} using projective representation theory. By different approaches, such surfaces were also found  by I. Dolgachev in \cite{dolgachev2018quartic}. The quartic surfaces $G$-invariant
by $G\in\{\mathfrak{A}_6,\mathfrak{S}_5\}$, were studied in \cite{faina2016s5}. In particular, the authors show that there are no $\mathfrak{A}_6$-invariants quartic surfaces. In \cite{faina2019Zp}, the authors determine for every $p \geq 5$,  all $\mathbb{Z}_p$-invariant non-singular quartic surfaces in $\mathbb{P}^3$ over  an algebraically closed field of characteristic zero.  The invariant  quartic surfaces by the Heisenberg group, $H\cong\mathbb{Z}_2^4$, were studied by D. Eklund  in \cite{eklund2018curves}.
In that work, the author finds all the $H$-invariant quartic surfaces and  shows that for a  generic of surfaces, its Picard number is 16.

Here we tackle the problem using the techniques presented in \cite{faina2016s5, faina2019Zp}. We develop, use, and implement a computational method, see  Appendix \ref{Program}. Using that method, we retrieve the results presented in \cite{faina2016s5} and in \cite{dolgachev2018quartic}. Beyond the previous results, we find the smooth invariant quartic surfaces for each one of the primitive subgroups of $\operatorname{Aut}(\mathbb{P}^3)\cong \operatorname{PGL}_4(\mathbb{C})$.

  The computational method presented in  Appendix \ref{Program} can be used to resolve the general version of the problem of finding the basis of the space of invariant polynomials. Nevertheless, the big obstacle is to understanding the finite subgroups $G$ of $\operatorname{PGL_n(\mathbb{C})}$, for $n\geq 4$, and to describe any projective faithful representation for each group. In particular, for Calabi-Yau hypersurfaces, it is more complicated because, in these cases, the degree and number of variables are not relatively primes. For that reason, the results presented by V. González-Aguilera, A. Liendo, and P. Montero \cite[Theo. 3.5]{gonzalez2020liftability} can not be applied.

The main result obtained in this work is the description of all the smooth quartic surfaces invariant for primitive groups of $\operatorname{PGL}_4(\mathbb{C})$.  Yet another important motivation for studying quartic surfaces is a longstanding conjecture by Burnside. 
Burnside established in a  conjectural way
that every smooth quartic surface has a group of projective  automorphisms of order bounded above by $2^4\cdot 120$, and that if a projective surface has a group of projective automorphisms of order $1920=2^4\cdot 120$, it must be projective equivalent to $\{ x_0^4 + x_1^4 + x_2^4 + x_3^4 + 12 x_0 x_1 x_2 x_3=0\}$, see \cite[Sect. 272, Ex. 6]{burnside1911theory}.  Kond\={o}  in \cite{kondo1999maximum}  showed that the maximum order  group of automorphisms of a K3 surface is $3840$. The maximum order is achieved only for the Kummer surface $\operatorname{Kum}(E_i\times E_i)$, where $E_i$ is the elliptic curve $\mathbb{C}/\langle 1, i \rangle$, and the group acting is a $\mathbb{Z}_4$ extension of the Mathieu Group $M_{20}\cong \mathbb{Z}_2^4\rtimes \mathfrak{A}_5$. Recently, Bonnafé and Sarti proved in \cite{bonnafe2019k3} that there are only two K3 surfaces admitting an action of  a group extension $G$ of $M_{20}$ with $|G|=1920$. Finally, by \cite[Remark 3.3]{bonnafe2019k3}, the Kond\={o} Kummer is not a quartic surface. On the other hand, in the same work is proved that their  examples (unique) are the smooth quartic surfaces $\{x_0^4+x_1^4+x_2^4+x_3^4+12x_0x_1x_2x_3=0\}$ and the Kummer surface $\operatorname{Kum}(E_{\tau}\times E_{2\tau})$, where $\tau=\frac{-1+i\sqrt{5}}{2}$ which can not be a quartic surface.

\begin{A}
The smooth quartic surfaces invariant by at least one finite primitive group of $\operatorname{PGL}_4(\mathbb{C})$ are listed as follows:

\begin{enumerate}
    \item  The   invariant  quartic smooth surface under the primitive action of the groups \\ $\mathbb{Z}_2^4.\mathfrak{S}_5, \mathbb{Z}_2^4.\mathfrak{A}_5, \mathbb{Z}_2^4.(\mathbb{Z}_5\rtimes\mathbb{Z}_4),\mathbb{Z}_4^2\rtimes D_{10},  $ and $\mathbb{Z}_4^2\rtimes \mathbb{Z}_5$ is 
    \vspace{0.2mm}
    $$ \{x_0^4+x_1^4+x_2^4+x_3^4+6 \left(x_0^2 x_1^2-x_2^2 x_1^2+x_3^2 x_1^2+x_0^2 x_2^2-x_0^2 x_3^2+x_2^2 x_3^2\right)=0\}.$$
    
    Moreover, the surface presented above is projectively equivalent to the Burnside quartic surface 
    \vspace{0.1mm}
    $$\{ x_0^4 + x_1^4 + x_2^4 + x_3^4 + 12 x_0 x_1 x_2 x_3=0\}.$$

    \item All the  quartic surfaces  invariant by the primitive representations in $\operatorname{PGL}_4(\mathbb{C})$ of $\mathfrak{A}_5$ are in the pencils of quartics
 
 \begin{enumerate}
     \item  The Pencil
     \[
\begin{split}
     \Big\{\lambda_0\Big( &  \frac{x_0^4}{2 \sqrt{3}}+\frac{1}{3} x_1 x_0^3+x_2 x_3 x_0^2+\frac{1}{3} x_1^3 x_0+\frac{1}{3} \sqrt{2} x_2^3 x_0+\sqrt{\frac{2}{3}} x_3^3 x_0 \\ & +\sqrt{\frac{2}{3}} x_1 x_2^3-\frac{1}{3} \sqrt{2} x_1 x_3^3+x_1^2 x_2 x_3-\frac{x_1^4}{2 \sqrt{3}}\Big)+ \\
    \lambda_1\Big(  & \frac{x_0^4}{6}+\frac{x_1 x_0^3}{\sqrt{3}}-\frac{1}{2} x_1^2 x_0^2+\sqrt{3} x_2 x_3 x_0^2+\frac{1}{3} \sqrt{2} x_3^3 x_0+x_1 x_2 x_3 x_0 \\ & +\frac{2}{3} \sqrt{2} x_1 x_2^3-\sqrt{\frac{2}{3}} x_1 x_3^3-\frac{1}{2} x_2^2 x_3^2-\frac{x_1^4}{3}\Big)=0 \Big\},
\end{split}
\]   
 where the singular  fibers are finitely many and at least  must include the fibers over the following points: 
 \[
\begin{split}
    \left[ - 2 \sqrt{2} : \sqrt{i \sqrt{15} -1} \right],\, \left[ - 2 \sqrt{2} : \sqrt{- i \sqrt{15} -1} \right], \,  \left[ 47 :8 \left(1-4 \sqrt{3}\right) \right]  \\
    \left[ 47 : - 8 \left(1 + 4 \sqrt{3}\right) \right], \, \left[ 4 : -\sqrt{3}+i \sqrt{5} \right], \;  \; \left[ 4 : -\sqrt{3} - i \sqrt{5} \right],   \text{ and } [1 : - \sqrt{3}]. 
\end{split}
\]

 \item The pencil
 \[
\begin{split}
    \Big\{\lambda_0& \Big( \frac{x_0^4}{\sqrt{15}}-\frac{x_2 x_3 x_0^2}{\sqrt{15}}-\frac{x_1^2 x_0^2}{2 \sqrt{15}}-\frac{1}{3} x_1^3 x_0-\frac{1}{3} x_2^3 x_0-\frac{1}{3} x_3^3 x_0+x_1 x_2 x_3 x_0  \\ & +\frac{1}{6} \sqrt{\frac{5}{3}} x_1^4-\frac{1}{3} \sqrt{\frac{5}{3}} x_1 x_2^3-\frac{1}{3} \sqrt{\frac{5}{3}} x_1 x_3^3+\frac{1}{2} \sqrt{\frac{5}{3}} x_2^2 x_3^2 \Big) + \\
   \lambda_1 & \left(\frac{x_0^4}{4}+\frac{1}{2} x_1^2 x_0^2+x_2 x_3 x_0^2+\frac{x_1^4}{4}+x_2^2 x_3^2+x_1^2 x_2 x_3 \right)=0 \Big\},
\end{split}
\]

%where the singular ones are the fibers on the eight points of $\mathbb{P}^{1}_{\lambda_0,\lambda_1}$
%$$[0:1],[6\sqrt{15}:1],[-\sqrt{15}:4],[-2\sqrt{15}:5],[15:-4\sqrt{15}],[6:-\sqrt{15}],[90:\sqrt{15}], \text{ and } [1:0].$$

where the singular ones are the fibers on the five points of $\mathbb{P}^{1}_{\lambda_0,\lambda_1}$
$$[0:1],\,[1:0],\, [\sqrt{15} : -4],\, [2 \sqrt{15} : -5], \, \text{and} \, [6 \sqrt{15} : 1] .$$

 \end{enumerate}

\item The smooth quartic surfaces  invariant by the primitive representations in $\operatorname{PGL}_4(\mathbb{C})$ of $\mathfrak{S}_5$  are in the pencil of quartics (b), plus the surfaces
 \[
\begin{split}
   \Big\{  &-\frac{x_0^4}{12}+\frac{x_1 x_0^3}{2 \sqrt{3}}-\frac{1}{2} x_1^2 x_0^2+\frac{1}{2} \sqrt{3} x_2 x_3 x_0^2+x_1 x_2 x_3 x_0-\frac{x_3^3 x_0}{3 \sqrt{2}}-\frac{x_1^3 x_0}{2 \sqrt{3}}-\frac{x_2^3 x_0}{\sqrt{6}} \\ & +\frac{x_1 x_2^3}{3 \sqrt{2}}-\frac{1}{2} x_2^2 x_3^2-\frac{1}{2} \sqrt{3} x_1^2 x_2 x_3-\frac{x_1^4}{12}-\frac{x_1 x_3^3}{\sqrt{6}}=0\Big\}, \text{ and }
\end{split}
\]

\[
\begin{split}
   \Big\{  &  \frac{x_0^4}{2 \sqrt{3}}+\frac{1}{3} x_1 x_0^3+x_2 x_3 x_0^2+\frac{1}{3} x_1^3 x_0+\frac{1}{3} \sqrt{2} x_2^3 x_0+\sqrt{\frac{2}{3}} x_3^3x_0 \\ & +\sqrt{\frac{2}{3}} x_1 x_2^3-\frac{1}{3} \sqrt{2} x_1 x_3^3+x_1^2 x_2 x_3-\frac{x_1^4}{2 \sqrt{3}}=0\Big\}.
\end{split}
\]

\item  The unique  smooth quartic surface  invariant by the primitive representations in $\operatorname{PGL}_4(\mathbb{C})$ of $\operatorname{PSL}_2(\mathbb{F}_7)$ is $\{2 x_0^4+ 6 x_1 x_2 x_3 x_0+x_1 x_3^3+x_1^3 x_2+x_2^3 x_3=0\}$.
\end{enumerate}

\end{A}

\begin{B}
    The  smooth quartic surface invariant by the biggest  group of projective automorphisms being finite primitive in $\operatorname{PGL}_4(\mathbb{C})$ is 
 $\{ x_0^4 + x_1^4 + x_2^4 + x_3^4 + 12 x_0 x_1 x_2 x_3=0\}$ (unique up to projective equivalence). Furthermore, its group of projective automorphisms is isomorphic to $\mathbb{Z}_2^4.\mathfrak{S}_5$ (with   order $1920$).

%Suppose that $X$ is a smooth quartic surface in $\PP^3$ such that $\operatorname{PAut}(X)$ is a finite primitive subgroup of $\PGL_4(\CC)$ of maximum order, i.e., if $Y$ is a smooth quartic surface of $\PP^3$ and $\operatorname{PAut}(Y)$ is a finite primitive subgroup of $\PGL_4(\CC)$, then $|\operatorname{PAut}(Y)| \leq |\operatorname{PAut}(X)|$. Then $X$ is projectively equivalent to $\{ x_0^4 + x_1^4 + x_2^4 + x_3^4 + 12 x_0 x_1 x_2 x_3=0\}$ and $\operatorname{PAut}(X)$ is isomorphic to $\mathbb{Z}_2^4.\mathfrak{S}_5$, which has  order $1920$.

\end{B}

\begin{remark}
In \cite{laza2022automorphisms}, the authors affirm  that the Kond\={o} surface $X_{\text{K\={o} }}=\operatorname{Kum}(E_i\times E_i)$ is the quartic $\{x_0^4+x_1^4+x_2^4+x_3^4+12x_0x_1x_2x_3=0\}$, which seen to be an inaccuracy because of the work of Bonnafé and Sarti, \cite[Remark 3.3]{bonnafe2019k3}.
\end{remark}

 \subsection*{Outline of the article} 

 The paper is organized as follows. In Section \ref{section:preliminaries} gathers notation, global conventions, and known results that will be used throughout the paper. In Section \ref{Section:=Theoretical } is presented the mathematics behind the computational algorithms.  Section \ref{Diagrams} contains relabeling of the diagrams presented in  \cite[Appendix A]{Ivan2019}. These diagrams show
 some inclusions between all finite primitive subgroups of $\PGL_{4}(\CC)$, up to conjugates.
 The description of all these groups can be seen in \cite[Chapter VII]{FiniteCollineationGroups}.  In Section \ref{PrimitiveGroups}, we hand over explicitly all the projective  primitive representations of the groups studied. In Section \ref{Snpg}, for the sake of completeness, we present all the  Non-Primitive projective representations of the groups $\mathfrak{A}_5,\mathfrak{S}_5,\mathfrak{A}_6,\, \mathbb{Z}_2^4 \rtimes \mathbb{Z}_5,\, \mathbb{Z}_2^4 \rtimes D_{10} \, \text{y} \, \operatorname{PSL}_2(\mathbb{F}_7)$. In
 Section \ref{QuarticInvariants}  we find for each group G in the previous two sections the collection of all (maximal)
$G$-invariant subspace of quartics. In Section \ref{Results}
   we list all the smooth quartic surfaces by the finite primitive groups of $\operatorname{PGL}_4(\mathbb{C})$. Finally, in Appendix \ref{Program}, we display the pseudocode \textit{\textbf{G-invariant forms of degree d}} implemented in the Program \texttt{Mathematica} \cite{Mathematica}, used to compute the invariant smooth quartic surfaces.

 %The paper is organized as follows. Section \ref{section:preliminaries} set the notation and review some results that we need for the sequel. In Section \ref{Section:=Theoretical } we present the mathematics behind the computational algorithms. In Section \ref{section:Representations} we hand over explicitly all the projective representation of the groups studied. In fact, it is organized as follows. First in \ref{Sub:= Symmetric and Alternating Groups} we present the projective representation of the groups $\mathbb{A}_5, \mathbb{A}_6$, and $\mathbb{S}_5$. Second, in \ref{subsection:PSL27} we provide the projective representation of the Projective Special Linear group $\operatorname{PSL}_2(\mathbb{Z}_7)$. Finally,  in \ref{subsection:SDP} we study the semidirect products $\mathbb{Z}_2^4\rtimes\mathbb{Z}_5$ and $\mathbb{Z}_2^4\rtimes D_{10}$. In section \ref{Section:= invariant}  using the result of sections \ref{Section:=Theoretical } and \ref{Section:= invariant}, we compute all the invariant smooth quartic surfaces for any of the concerning groups. Finally, in Apendix \ref{Program} we display the pseudocodes \textit{\textbf{InvariantForms}[G,d]}, and \textit{\textbf{MacaulayDiscriminant}[f,d]}  implemented in the Program \texttt{Mathematica} \cite{Mathematica}, used to compute the invariant smooth quartic surfaces.

\subsection*{Acknowledgments}  The authors thank Universidad del Valle for the hospitality with the third author
during his stay in Cali. We sincerely thank Pedro Montero for introducing us  the problem and for his
valuable discussions. We also want to thank Arnaud Beauville and Yuri Tschinkel for suggesting we 
see the works of Ivan Cheltsov and Constantin Shramov during the stay of the third author in Roma on
the occasion of Sandro Verra’s 70th birthday. Finally, we thank Pablo Quezada Mora (student of Paola Comparin) for the valuable conversations about the topic in the event ENEMAT 2022.
The results of this work are part of the undergraduate
Thesis of the first Author. 
\ 
The third author was partially supported by Fondecyt Posdoctorado  ANID project 3210518.

\section{Notations and Preliminaries}\label{section:preliminaries}
For the reader's convenience and to set up the notation, we now briefly recall the needed notions.

\textbf{Some Constants}
\begin{itemize}
    \item[$\omega$:] Primitive $3$rd root of unity, $\omega^3 = 1$.
    \item[$i$:] Primitive $4$-th root of unity, $i^2=-1$.
    \item[$\mu$:] Primitive $5$-th root of unity, $\mu^5 = 1$. 
    \item[$\zeta$:] Primitive $7$-th root of unity, $\zeta^7=1$. 
    \item[$\psi$:] Primitive $8$-th root of unity, $\psi^2 = i$. 
    %\item[$b_{\pm}$:] $(1 \pm \sqrt{5})/4$.
    \item[$s_{\pm}$:] $\zeta^2 \pm \zeta^5$.
    \item[$t_{\pm}$:] $\zeta^4 \pm \zeta^3$.
    \item[$u_{\pm}$:] $\zeta \pm \zeta^6$. 
    \item[$\lambda_{\pm}$:] $(-1 \pm i \sqrt{15})/4$. 
    \item[$\nu_{\pm}$:] $(\sqrt{3} \pm i \sqrt{5})/ 2 \sqrt{2}$. 
\end{itemize}

\textbf{Conventions and Notations}

\begin{itemize}
    \item $n$ always denote an integer $\geq 2$ and $d$ always denote an integer $\geq 3$,
    \item For us $\mathbb{Z}_n$ denotes the cyclic group of order $n$,
    \item $\mathbb{F}_q$ is the finite field of $q=p^a$ elements,
    \item The alternating and symmetric groups in $n$ letter are denoted by $\mathfrak{A}_n$ and $\mathfrak{S}_n$ respectively,
    \item A dihedral group of order $n$ is denoted by $D_n$. In particular, $D_{10}$ is the dihedral group of $10$ elements.
    \item Given a group $G$, here we denote $\widehat{G\times G}$ the group $(G\times G)\rtimes \mathbb{Z}_2 $, where the action is given exchanging factors. Similarly, the groups $\widehat{(\mathfrak{A}_4 \times \mathfrak{A}_4)\rtimes \ZZ_2^{(i)}}$ in  Diagram \ref{fig:1° - 12°Isomorphic} denote a extension of these groups by an element of order two which exchange the factors of $\mathfrak{A}_4 \times \mathfrak{A}_4$.
    \item Here $H.G$ denoted a nonsplit extension of $G$ by $H$, i.e., there is a short exact sequence of groups $1\to H\to H.G\to G\to 1$ and $H.G$ is not isomorphic to $H\times G$,
    \item The groups $\operatorname{GL_n(\mathbb{C})}$ and $\operatorname{PGL_n(\mathbb{C})}$ are the general linear group of degree $n$ over the complex number $\mathbb{C}$ and its projectivization.
    \item The group $\operatorname{SL}_n(\CC)$ is the special linear group of degree $n$ over complex numbers.
    \item $E_n$ is the identity matrix of size $n \times n$. 
\end{itemize}

A hypersurface of the complex projective space $\mathbb{P}^n$ is defined as the zero locus of a degree $d$ form $f$ in the polynomial ring $\mathbb{C}[x_0,\dots,x_n]$. A good reference to introduce into the study of hypersurfaces is presented by J. Kollár in \cite{kollar2019algebraic} or by O. Debarre in \cite{debarre2017geometry}. Also, I. Dolgachev presented a concise introduction to automorphisms of algebraic varieties in his notes written  for the occasion of the ``Fourth Latin American School on Algebraic Geometry and its Applications'' ELGA, in Talca, Chile, \cite{dolgachev2019brief}. 

We denote by $\operatorname{Form}_{n,d}$ to the subset of $\CC[x_0,\ldots,x_n]$ of all homogeneous polynomials of degree $d$ in the variables $x_0,\ldots,x_n$, and the zero polynomial. It is clear $\operatorname{Form}_{n,d}$ is a linear subspace of $\CC^{n+1}$ of finite dimension given by $\binom{n+d}{n}$.

For $f,g \in \CC[x_0,\ldots,x_n]$ we write $f \sim g$ if  $f = \lambda g$ for some $\lambda \in \CC^*$.

For $f \in \CC[x_0,\ldots,x_n]$ and $A \in \GL_{n+1}(\CC)$ we define $Af$ as the polynomial in  $\CC[x_0,\ldots,x_n]$ given by
\[
Af(x) = f(Ax),
\]
where 
\[
Ax = A 
\begin{bmatrix}
 x_0 \\
 x_1 \\
 \vdots \\
 x_n
\end{bmatrix}. 
\]

For $V$ a subspace of $\CC^{n+1}$, and matrix $(A) \in \PGL_{n+1}(\CC)$ we write $(A)V = \{ A v \, : \, v \in V \}$. 

\begin{definition} Recall that a hypersurface of  $\mathbb{P}^n$ is the zero locus of a single irreducible polynomial $f\in \mathbb{C}[x_0,\dots,x_n]$, i.e., $X=\{f=0\}$.
    Let $X$ and $Y$ be hypersurfaces in $\PP^n$. We say $X$ and $Y$ are projectively equivalent if and only if $X = (A)Y$ for some $(A) \in \PGL_{n+1}(\CC)$. 
\end{definition}

\begin{definition}
Let $X$ be a hypersurface in $\PP^n$ and $f \in \CC[x_0,\ldots,x_n]$ homogeneous non-constant polynomial. Define the following subgroups of $\PGL_{n+1}(\CC)$:
\[
\begin{split}
    \operatorname{PAut}(X) & =  \{ (A) \in \PGL_{n+1}(\CC) \, : \, (A)X = X \}, \\
    \operatorname{PAut}(f) & = \{ (A) \in \PGL_{n+1}(\CC) \, : \, Af \sim f \}.
\end{split}
\]
\end{definition}

\begin{example}
Let $$A=\begin{bmatrix}
a & 0 & 0 \\
0 & b & 0 \\
0 & 0 & c
\end{bmatrix}\in \operatorname{GL}_3(\mathbb{C}),$$ and consider the forms of degree two in $\mathbb{C}[x,y,z]$:
\[
f(x,y,z) = x^2+y^2-z^2, \quad g(x,y,z) = xz+yz+xy \quad \text{and} \quad h(x,y,z) = x^2+yz. \]
We see that $(A)\{f=0\}=\{f=0\}$ if and only if $f\sim Af$, then
$Af(x,y,z)=a^2x^2+b^2 y^2-c^2z^2$, so, the entries of the matrix $A$ must satisfy the relations $a^2=b^2=-c^2$, so $A=\lambda D$, where $\lambda\in \mathbb{C}^\times$ and $D$ is a diagonal matrix with entries $1$ or $\pm i$.
Similarly, if we do an analysis for the degree two form $g(x,y,z)=xz+yz+xy$, we get that $(A)\{g=0\} = \{g = 0 \}$ if and only if $Ag(x,y,z)=(ac)xz+(bc)yz+(ab)xy$, and so the equality $g= \lambda \cdot Ag$ implies that $a=b=c$, therefore $A=\lambda E_3$. Finally, we leave it to the reader to verify that the conditions that must to satisfy the entries of the matrix $A$ to get that $h \sim Ah$ are given by $$\operatorname{det}\begin{bmatrix}
a & b \\
c & a
\end{bmatrix}=0.$$ 
\end{example}

\begin{definition}\label{G-invariant}
    Let $X = \{ f = 0 \}$ be a hypersurface in $\PP^n$,  and let $g \in \CC[x_0,\ldots,x_n]$ be a homogeneous non-constant polynomial, and $G$ be a subgroup of $\PGL_{n+1}(\CC)$. We say $X$ is \textit{$G$-invariant} if and only if $(A)X = X$ for all $(A) \in G$, i.e., $G \subset \operatorname{PAut}(X)$. Likewise, we say $f$ is \textit{$G$-invariant} if and only if $Af \sim f$ for all $(A) \in G$, i.e., $G \subset \operatorname{PAut}(f)$.
\end{definition}

\begin{definition}
Let $V$ be a subspace of $\operatorname{Form}_{n,d}$ and $G$ be a subgroup of $\PGL_{n+1}(\CC)$. We say $V$ is a \textit{$G$-invariant subspace} of $\operatorname{Form}_{n,d}$ if and only if $ Af \sim f$ for all $f \in V$ and $(A) \in G$, i.e., $G \subset \operatorname{PAut}(f)$ for all $f \in V$. 
\end{definition}

\begin{definition}
Given $G$  a subgroup of $\PGL_{n+1}(\CC)$, a $G$-\textit{decomposition} of $\CC^{n+1}$ is a set $\Gamma$ of non-trivial subspaces of $\CC^{n+1}$ such that
\begin{itemize}
    \item[\textsc{i)}] $|\Gamma| \geq 2$.
    \item[\textsc{ii)}] $\CC^{n+1} = \bigoplus_{V \in \Gamma} V$.
    \item[\textsc{iii)}] For each $(A) \in G$ and $V \in \Gamma$ we have $(A)V = W$ for some $W \in \Gamma$. 
\end{itemize}
\end{definition}

\begin{definition}
A subgroup $G$ of $\PGL_{n+1}(\CC)$ is called  primitive if and only if there is no  $G$-decomposition of $\CC^{n+1}$. 
\end{definition}

\begin{definition}
    Let $G$ be a non-primitive subgroup of $\PGL_{n+1}(\CC)$. Then $G$ is called imprimitive if and only if for each $\Gamma$ a $G$-decomposition of $\CC^{n+1}$ and $V,W \in \Gamma$ there  exists $(A) \in G$ such that $(A)V = W$. 
\end{definition}

\begin{definition}
    Let $G$ be a subgroup of $\PGL_{n+1}(\CC)$. We say $G$ is \textit{intransitive} if and only if $G$ is neither primitive nor imprimitive.
\end{definition}
\begin{example}

\begin{enumerate}
    \item For $n \geq 2$ let $G = \langle (A) \rangle$ be the subgroup of $\PGL_n(\CC)$ generated by\\ $A \in \GL_n(\CC)$. We have $G$ is intransitive since every invertible matrix is diagonalizable. In particular, for $n = 4$,
\[
A = \begin{bmatrix}
0 & 1 & 0 & 0 \\
1 & 0 & 0 & 0 \\
0 & 0 & 0 & 1 \\
0 & 0 & 1 & 0
\end{bmatrix},
\]
and $G = \langle (A) \rangle \subset \PGL_4(\CC)$, we have the $G$-decomposition of $\CC^4$: $\Gamma = \{V_1,V_2\}$, where $V_1$ is generated by $e_1,e_2$ and $V_2$ is generated by $e_3,e_4$, as  usual $e_j$ denotes the $j$-th column of identity matrix $E_4$. 

\item On the other hand, 
consider the representation of $\operatorname{PSL}_2(\mathbb{F}_7)$ given in Remark \ref{klein}, denoted by $\overline{G}_{168}$ in the notation used in \cite{Ivan2019}. It is clear that $\overline{G}_{168}$ cannot be intranstive since $\overline{G}_{168}$ does not fix any point on $\PP^2$. If $\overline{G}_{168}$ is imprimitive, then we find a non-trivial homomorphism of groups $\varphi:\overline{G}_{168} \to \mathfrak{S}_3$, and since $\overline{G}_{168}$ is simple, then $\varphi$ is injective and we get a contradiction by the orders of $\overline{G}_{168}$ and $\mathfrak{S}_3$. Therefore, $\overline{G}_{168}$ is primitive.

\end{enumerate}

%consider $G$ a finite simple irreducible subgroup of $\PGL_3(\CC)$ of order greater than $3$ such that $\overline{G} \cong G$, where $\overline{G} = \{ (A) \mid A \in G \}$. For example,

\end{example}
\section{Theoretical Invariant Subspaces}\label{Section:=Theoretical }

%Theoretical Aspects of the Computational Method

Let  $G$ be a subgroup of $\PGL_{n+1}(\CC)$. By the preliminaries, the problem of finding the smooth hypersurfaces of degree $d$ invariant by $G$ is equivalent to finding the homogeneous non-singular forms of degree $d$ invariant by $G$. This task can be done in two parts: 
\begin{itemize}
    \item[\textsc{i)}] To find all homogeneous forms $f \in \operatorname{Form}_{n,d}$ such that $f \sim Af$ for all $(A) \in G$, and 
    \item[\textsc{ii)}] determinate which of such $f$ are non-singular.
\end{itemize}

Here we will present a  solution to the first problem when $G$ is finite. Moreover, the solution is computable. In fact, we implemented an algorithm in \texttt{Mathematica} \cite{Mathematica} to solve the item \textsc{i)} for a given finite subgroup of $\PGL_4(\CC)$, see Appendix \ref{Program}.

With respect to item \textsc{i)}, we relate the   concepts of \textit{$G$-invariant subspace of $\operatorname{Form}_{n,d}$}, and \textit{maximal $G$-invariant subspace of $\operatorname{Form}_{n,d}$}. It turns out to be a finite number of such \textit{maximal $G$-invariant subspace of $\operatorname{Form}_{n,d}$}, see Corollary \ref{FiniteMaximalGInvariantSubspaces}, and the description of these subspaces is straightforward given a set of generators of $G$, see Theorem \ref{PrincipalInvariantSubspaceTheorem}. Therefore, Corollary \ref{Correspondence} concludes the item \textsc{i)}.

Concerning item \textsc{ii)}, we recall basic results of the theory of resultants and elimination. In particular the \textit{discriminant} of $f$ solves the question of determinating if $f$ is neither singular nor non-singular. However, the computation of the \textit{discriminant} could be an exaggeration to solve this question in many cases, so we will state the Lemma \ref{MSC} which is much easier to test on any homogeneous form, although this result is not an equivalent criterion of non-singularity.

\subsection{Invariant Subspaces}\label{Invariant Subspaces}

\begin{lemma}\label{LemaMaximalGSubspaces}
Let $V$ be a $G$-invariant non-trivial subspace of $\operatorname{Form}_{n,d}$, $(A) \in G$ and $p \in V$ non-zero. If  we have that $Ap=\lambda p$ for some unique $\lambda \in \CC^*$, then $Af = \lambda f$ for all $f \in V$.
\end{lemma}

\begin{proof}
Clearly, $Af = \lambda f$ if $f = 0$. Suppose then that  $f \neq 0$, we have $f,f+p \in V$, then $Af = af$ and $A(f+p) = b (f+p)$ for some $a,b \in \CC^*$. Thus
$
bf+bp = b(f+p) = A(f+p) = Af+Ap = af + \lambda p
$ 
and so $(b-a)f=(\lambda-b)p$. If $b-a = 0$, then $(\lambda - b) p = 0$ and $\lambda=b=a$. On the other hand, if $b-a \neq 0$, then $f = \tau p$ for some $\tau \in \CC^*$ and 
$
af=Af = A(\tau p) = \tau(Ap) = \tau \lambda p = \lambda f,
$
thus $a = \lambda$. In any case we get that $a = \lambda$. 
\end{proof}

\begin{theorem}
Let $V,W$ be two $G$-invariant subspaces of $\operatorname{Form}_{n,d}$. Then $V \cap W \neq 0$ implies $V + W$ is a $G$-invariant subspace of $\operatorname{Form}_{n,d}$.
\end{theorem}
\begin{proof}
There is nothing to prove if $V$ or $W$ is trivial. Therefore, we can suppose that $V$ and $W$ are non-trivial, and take $p \in V \cap W$ non-zero polynomial. For $(A) \in G$ we have that $Ap = \lambda p$ for some $\lambda \in \CC^*$. By the previous Lemma \ref{LemaMaximalGSubspaces}, we see that $Af = \lambda f$, and $Ag = \lambda g$ for all $f \in V$ and $g \in W$. 

If $f+g \in V + W$, $f \in V$ and $g \in W$, then 
$
A(f+g) = Af + Ag = \lambda f + \lambda g = \lambda(f+g)
$
and so $A(f+g) \sim f+g$.
\end{proof}

\begin{corollary}
If $f \in \operatorname{Form}_{n,d}$ is a non-zero polynomial $G$-invariant, then there exists a unique $V$ maximal $G$-invariant subspace of $\operatorname{Form}_{n,d}$ such that $f \in V$. 
\end{corollary}

\begin{corollary}\label{FiniteMaximalGInvariantSubspaces}
There is a finite number of maximal $G$-invariant subspaces of $\operatorname{Form}_{n,d}$. 
\end{corollary}

By the preliminaries results we obtain the correspondence between Smooth surfaces of degree $d$ contained in $\mathbb{P}^n$ and the non-singular forms contained in the collection of all the maximal $G$-invariant subspaces of $\operatorname{Form}_{n,d}$. In sum, we have Corollary \ref{Correspondence}.

\begin{corollary}\label{Correspondence}
Let $\Gamma$ be the finite collection of all maximal $G$-invariant subspaces of $\operatorname{Form}_{n,d}$. Then we have a correspondence between
\[
\begin{tikzcd}
\{ \textit{Smooth hypersurfaces of degree $d$ in $\PP^n$} \} \arrow[rr, Leftrightarrow] &  & \left\{\textit{Non-singular forms in $\bigcup \Gamma$} \right\}.
\end{tikzcd}
\]
\end{corollary}

In the rest of this section we  suppose that $G$ is a finite group. We are going to construct all maximal $G$-invariant subspaces of $\operatorname{Form}_{n,d}$. In first place, let's take $\{A_i\}_{1 \leq i \leq m} \subset \SL_{n+1}(\CC)$ such that $G$ is generated by $\{(A_i) \}_{1 \leq i \leq m}$. Let $\sigma_i$ be the order of $(A_i)$, and $\kappa \geq 1$ such that $\kappa d$ is the least common multiple of $n+1$ and $d$, i.e., $\operatorname{lcm}(n+1,d)=\kappa d$. Thus, $A_i^{\sigma_i} = \tau_i E_{n+1}$ for some $\tau_i \in \CC^*$, since $\det(A_i) = 1$, then $\tau_i^{n+1} = \det(\tau_i E_{n+1}) = \det(A_i^{\sigma_i}) = 1,\, \tau_i^{n+1} = 1$. Now  let's define 
\[
K = \{ k \in \CC^m \mid k_i^{\kappa \sigma_i} = 1,\, \text{for each}\, 1 \leq i \leq m \}.
\]

For each $1 \leq j \leq m$, and  for any $k \in K$ denote by $B_{k,j}$ the linear transformation over $\text{Form}_{n,d}$ given by the formula
\[
B_{k,j}f = A_j f- k_j f.
\]

For each $k \in K$, we define $T_k \colon \operatorname{Form}_{n,d} \to \operatorname{Form}_{n,d}^m$ by $(T_k g)_j = B_{k,j} g$, for $1 \leq j \leq m$, and $g \in \operatorname{Form}_{n,d}$. It is easy to see  that $T_k$ is a linear map since each $B_{j,k}$ is a linear transformation. Let
\[
\Gamma = \{ \ker(T_k) \, : \, k \in K, \, \ker(T_k) \neq 0 \}
\]
be a collection, possibly empty, of non-trivial subspaces of $\operatorname{Form}_{n,d}$. 

\begin{lemma}\label{FirstLemmatoPrincipalTheoremInvariantSubspaces}
If $f \in \operatorname{Form}_{n,d}$ is $G$-invariant, then $f \in \ker(T_k)$ for some $k \in K$.
\end{lemma}

\begin{proof}
If $f = 0$, then $f \in \ker(T_k)$ for all $k \in K$. On the other hand, if $f \neq 0$, then $A_if = \lambda_i f$ for some $\lambda_i \in \CC^*$. Thus,
$
\lambda_i^{\sigma_i} f = A_i^{\sigma_i} f = (\tau_i E_{n+1}) f = \tau_i^d f,
$ 
$\lambda_i^{\sigma_i} = \tau_i^d,\, \lambda_i^{\kappa \sigma_i} = \tau_i^{\kappa d} = 1,\, \lambda_i^{\kappa \sigma_i} = 1$. Therefore, $\lambda =(\lambda_i) \in K$ and $f \in \ker(T_\lambda)$.
\end{proof}

\begin{lemma}\label{SecondLemmatoPrincipalTheoremInvariantSubspaces}
The subspace $\ker(T_k)$ is a $G$-invariant subspace of $\operatorname{Form}_{n,d}$ for all $k \in K$.
\end{lemma}
\begin{proof}
If $f \in \ker(T_k)$, then $A_i f = k_i f$ for all $i$, so $\{(A_i)\}_{1 \leq i \leq m } \subset \operatorname{PAut}(f)$. Since $G$ is generate by $\{(A_i)\}_{1 \leq i \leq m }$, then $G \subset \operatorname{PAut}(f)$, so $\ker(T_k)$ is a $G$-invariant subspace of $\operatorname{Form}_{n,d}$.
\end{proof}

\begin{lemma}\label{ThirdLemmatoPrincipalTheoremInvariantSubspaces}
If $\ker(T_k) \cap \ker (T_\lambda) \neq 0$, then $k = \lambda$ and $\ker(T_k) = \ker(T_\lambda)$. 
\end{lemma}
\begin{proof}
Suppose $p \in \ker(T_k) \cap \ker(T_\lambda)$ is a non-zero element, then $T_k(p) = 0$ and $T_\lambda(p) = 0$, this implies $A_i p = k_i p$ and $A_i p = \lambda_i p$ for all $i$, so $k_i p = \lambda_i p$, $k_i = \lambda_i$, $k = \lambda$ and $\ker(T_k) = \ker(T_\lambda)$.
\end{proof}

\begin{theorem} \label{PrincipalInvariantSubspaceTheorem} Let $\Gamma$ as in Corollary \ref{Correspondence}.
If $\Gamma \neq \emptyset$, then $\Gamma$ is the set of all maximal $G$-invariant subspaces of $\operatorname{Form}_{n,d}$. Otherwise, if $\Gamma = \emptyset$, then the trivial space is the unique $G$-invariant subspace of $\operatorname{Form}_{n,d}$.
\end{theorem}

\begin{proof}
Suppose $\Gamma = \emptyset$, then we want to prove there is not  a non-zero $G$-invariant polynomial in $\operatorname{Form}_{n,d}$. If $f \in \operatorname{Form}_{n,d}$ is a $G$-invariant non-zero polynomial, then by Lemma \ref{FirstLemmatoPrincipalTheoremInvariantSubspaces} we have $f \in \ker(T_k)$ for some $k \in K$, so $\ker(T_k) \neq 0$ and $\ker(T_k) \in \Gamma$, which gives a contradiction.

Now we want to prove that $V$ is a non-trivial $G$-invariant  subspace of $\operatorname{Form}_{n,d}$ if and only if $V \in \Gamma$. Let $V$ be a maximal $G$-invariant non-trivial subspace of $\operatorname{Form}_{n,d}$, and let $p \in V$ be non-zero element. Then by Lemma \ref{FirstLemmatoPrincipalTheoremInvariantSubspaces} we have that $p \in \ker(T_k)$ for some $k \in K$, then $A_i p = k_i p$ for all $i$. If $f \in V$  by Lemma \ref{LemaMaximalGSubspaces} we have that  $A_i f = k_i f$ for all $i$, thus, $f \in \ker(T_k)$ and $V \subset \ker(T_k)$. However, $\ker(T_k)$ is a $G$-invariant subspace of $\operatorname{Form}_{n,d}$ because of Lemma \ref{SecondLemmatoPrincipalTheoremInvariantSubspaces}, thus $V = \ker(T_k)$ and $V \in \Gamma$. Then $\ker(T_k) \in \Gamma$ is contained in some $\ker(T_\lambda) \in \Gamma$ maximal $G$-invariant subspace of $\operatorname{Form}_{n,d}$. By Lemma \ref{ThirdLemmatoPrincipalTheoremInvariantSubspaces}, we have that $\lambda = k$, and $\ker(T_k)$ is a maximal $G$-invariant subspace of $\operatorname{Form}_{n,d}$. 
\end{proof}

The relevance of Theorem \ref{PrincipalInvariantSubspaceTheorem} is that it provides  an algorithm to find all maximal $G$-invariant subspace of $\operatorname{Form}_{n,d}$. Indeed, as the set of generators of $G$ is finite since $G$ is finite, then $K$ is finite too, and the linear maps $T_k$ have a finite size matrix representations. With such matrices we find the kernels of $T_k$ by software computations. Under these considerations we implemented the algorithm \textbf{InvariantForms}[\textit{G},\textit{d}] in \texttt{Mathematica} \cite{Mathematica}, see  Appendix \ref{Program}. 
 
\subsection{Non-Singularity}

In this section we recall some well-known facts in theory of resultants and elimination; the following references give a basic background in this matter \cite{Cox05,Gelfand94,vanWaerdenII}. Furthermore, we present a result by  K. Oguiso, and X. Yu which says that under certain conditions a homogeneous polynomial is singular, see Lemma \ref{MSC}.

The discriminant of a homogeneous form of degree $d$, $f \in \operatorname{Form}_{n,d}$, which we denote by $\operatorname{Disc}_d(f)$ or $\operatorname{Disc}(f)$, is a homogeneous polynomial in the coefficients $f$ of degree $(n+1)(d-1)^n$ such that $\operatorname{Disc}(f) \neq 0$ if and only if $f$ is non-singular.  

%\textcolor{red}{Aquí he agregado el grado homogeneo del discriminante, esto lo usare en en lo que hace falta de la formula, para justificar que hay finitos puntos.   }\textcolor{blue}{la finitud es bien sabida por los geometras ya que los puntos con preimagen mala forman un cerrado zariski de una curva, es decir una colección finita de puntos}

%El discriminante de una forma homogénea $f \in \CC[x_0,\ldots,x_n]$ de grado $d$, que denotamos por $\operatorname{Disc}_d(f)$ o $\operatorname{Disc}(f)$, es un polinomio en los coeficientes de $f$ tal que $\operatorname{Disc}(f) \neq 0$ si, y solo si, $f$ es no-singular. 

For example, if $n=1$, $d=2$, and $f = a x_0^2 + b x_0 x_1 + c \in \CC[x_0,x_1]$, then $\operatorname{Disc}(f) = 4ac - b^2$ and we have $f$ is non-singular if and only if $4ac-b^2 \neq 0$.

It is well-known that for fixed $n$ and $d$, the discriminant $\operatorname{Disc}(\cdot)$ exists as a polynomial function defined in $\operatorname{Form}_{n,d}$. The problem is how to compute it. From theory of Resultants we know $\operatorname{Disc}(f)$ is the resultant of all partial derivates of $f$. We follow the Macaulay's construction of this resultant to find the discriminant $\operatorname{Disc}(f)$.

The Macaulay's method: Consider two square matrices $Q(f)$ and $Q'(f)$, where the latter is a \textit{submatrix} of the former, such that
$
\det(Q(f)) = \operatorname{Disc}(f) \cdot \det(Q'(f)).
$
Thus, if $\det(Q'(f)) \neq 0$, we have the expression
\[
\operatorname{Disc}(f) = \frac{\det(Q(f))}{\det(Q'(f))}. 
\]

To overcome the difficulty occasioned by the possibility of vanishing of $\det(Q'(f))$, we can  take an element  $\varphi \in \operatorname{Form}_{n,d}$ such that $\det(Q'(\varphi)) \neq 0$. e.g., 
%for example
for the Fermat polynomial of degree $d$, $\varphi =  d^{-1}(x_0^d + x_1^d + \cdots + x_n^d)$, we have that $Q(\varphi)$ and $Q'(\varphi)$ are identity matrices. Now, if $\lambda$ and $\mu$ are indeterminates we have that
\[
\det(Q(\lambda f + \mu \varphi)) = \operatorname{Disc}(\lambda f + \mu \varphi) \cdot \det(Q'(\lambda f + \mu \varphi)).
\]
For fixed $f$, this expression is an identity of polynomials in the variables $\lambda$ and $\mu$, and $\det(Q'(\lambda f + \mu \varphi))$ is not the zero polynomial since for $\lambda = 0$ and $\mu = 1$ we have $\det(Q'(\lambda f + \mu \varphi)) = \det(Q'(\varphi))=1$. Therefore $\det(Q'(\lambda f + \mu \varphi))$ divides $\det(Q(\lambda f + \mu \varphi))$ as polynomials in the variables $\lambda$ and $\mu$, and we get the expression
\[
\operatorname{Disc}(f) = \left. \frac{\det(Q(\lambda f + \mu \varphi))}{\det(Q'(\lambda f + \mu \varphi))} \right|_{(\lambda,\mu) = (1,0)}
\]
of discriminant of $f$, which it is valid in general.

%We use and recommend the software \textit{\textbf{Macaulay2}} \cite{Macaulay2} to compute discriminants, see \cite{Staglianò18}. 

% We implement the algorithm \textbf{MacaulayDiscriminant}[\textit{f},\textit{d}] in \texttt{Mathematica} \cite{Mathematica}, see Apendix \ref{Program} to compute the discriminant of $f$ \textit{à la Macaulay}.

Certainly, in many cases it could be an exaggeration to compute the discriminant of $f$ to determinate if $f$ is non-singular. The following lemma gives us some sufficient conditions for singularity of a homogeneous form, which it is proved in \cite[Section 3, Lemma 3.2, Proposition 3.3]{oguiso2015automorphism}.

\begin{lemma}\label{MSC}
Let $\{ f = 0 \}$ a hypersurface of degree $d$ given by the homogeneous form $f \in \operatorname{Form}_{n,d}$ . Let $a,b$ two nonnegative integers such that $2a+b \leq n$. If there are $a+b$ distinct variables $x_{i_1},\ldots,x_{i_{a+b}}$ such that $f \in (x_{i_1},\ldots,x_{i_a}) + (x_{i_{a+1},\ldots,x_{i_{a+b}}})^2$, then $\{ f = 0 \}$ is singular. In particular, if $\{ f = 0 \}$ is a smooth hypersurface, then for every $i \in \{1,\ldots,n+1\}$, the monomial $x_i^{d-1}x_j$ must appear with non-zero coefficient in the form $f$ for some $j$. 
\end{lemma}

The last statement  is useful because it can be easily tested for any $f \in \operatorname{Form}_{n,d}$. 

\section{Diagrams}\label{Diagrams}
Here, we present a relabeling of the diagrams presented in  \cite[Appendix A]{Ivan2019}. These diagrams show
 some inclusions between all finite primitive subgroups of $\PGL_{4}(\CC)$, up to conjugates.
 The description of all these groups can be viewed in \cite[Chapter VII]{FiniteCollineationGroups}. In the next section, we will present a set of generators for some of these groups.
\begin{figure}[H]
\centering
\begin{minipage}{.5\textwidth}
  \centering
  \begin{tikzcd}
    12^\circ                                 &                                                     &                               &                              \\
                                         & 5^\circ \arrow[lu]                                  & 11^\circ                      &                              \\
    8^\circ \arrow[uu]                       & 9^\circ \arrow[luu]                                 & 6^\circ                       & 7^\circ \arrow[lu]           \\
    2^\circ \arrow[u] \arrow[ruu] \arrow[ru] & 10^\circ \arrow[u] \arrow[ruu] \arrow[lu]           & 3^\circ \arrow[u] \arrow[luu] & 4^\circ \arrow[u] \arrow[lu] \\
                                         & 1^\circ \arrow[rru] \arrow[ru] \arrow[lu] \arrow[u] &                               &                             
    \end{tikzcd}
    \captionof{figure}{1° - 12°}
    \label{fig:1° - 12°}
\end{minipage}%
\begin{minipage}{.5\textwidth}
  \centering
  \begin{tikzcd}
                              & 21^\circ                                 &                               \\
        18^\circ \arrow[ru]           & 20^\circ \arrow[u]                       & 19^\circ  \arrow[lu]           \\
        16^\circ \arrow[ru] \arrow[u] & 15^\circ \arrow[lu] \arrow[ru]           & 17^\circ \arrow[lu] \arrow[u] \\
                              & 14^\circ \arrow[lu] \arrow[ru] \arrow[u] &                               \\
                              & 13^\circ \arrow[u]                       &                              
    \end{tikzcd}
    \captionof{figure}{13° - 21°}
    \label{fig:13° - 21°}
\end{minipage}
\end{figure}

\begin{figure}[H]
\centering
\begin{minipage}{.5\textwidth}
  \centering
  \begin{tikzcd}
               & (F)                       &                           &     &                \\
               & (K) \arrow[u]             &                           & (D) &                \\
    (G) \arrow[ru] &                           & (C) \arrow[lu] \arrow[ru] &     & (E) \arrow[lu] \\
               & (A) \arrow[lu] \arrow[ru] &                           &     &               
    \end{tikzcd}
    \captionof{figure}{(A), (C) - (F), (G) y (K)}
    \label{fig:(A), (C) - (F), (G) y (K)}
\end{minipage}%
\begin{minipage}{.5\textwidth}
  \centering
    \begin{tikzcd}
    (H)           \\
    (B) \arrow[u]
    \end{tikzcd}
    \captionof{figure}{(B) y (H)}
    \label{fig:(B) y (H)}
\end{minipage}
\end{figure}

The following diagrams point out the isomorphic class to which each one of these groups belongs.

\begin{figure}[H]
    \centering
    \begin{tikzcd}
\widehat{\mathfrak{S}_4 \times \mathfrak{S}_4}                                       &                                                                                  &                                                            &                                                           \\
                                                                                     & \mathfrak{S}_4 \times \mathfrak{S}_4 \arrow[lu]                                  & \widehat{\mathfrak{A}_5 \times \mathfrak{A}_5}             &                                                           \\
\widehat{(\mathfrak{A}_4 \times \mathfrak{A}_4)\rtimes \ZZ_2^{(1)}} \arrow[uu]                 & \widehat{(\mathfrak{A}_4 \times \mathfrak{A}_4)\rtimes \ZZ_2^{(2)}} \arrow[luu]            & \mathfrak{S}_4 \times \mathfrak{A}_5                       & \mathfrak{A}_5 \times \mathfrak{A}_5 \arrow[lu]           \\
(\mathfrak{A}_4 \times \mathfrak{A}_4)\rtimes \ZZ_2 \arrow[u] \arrow[ruu] \arrow[ru] & \widehat{\mathfrak{A}_4 \times \mathfrak{A}_4} \arrow[u] \arrow[ruu] \arrow[lu]  & \mathfrak{A}_4 \times \mathfrak{S}_4 \arrow[u] \arrow[luu] & \mathfrak{A}_4 \times \mathfrak{A}_5 \arrow[u] \arrow[lu] \\
                                                                                     & \mathfrak{A}_4 \times \mathfrak{A}_4 \arrow[rru] \arrow[ru] \arrow[lu] \arrow[u] &                                                            &                                                          
\end{tikzcd}
    \caption{1° - 12°}
    \label{fig:1° - 12°Isomorphic}
\end{figure}

\begin{figure}[H]
\centering
\begin{minipage}{.5\textwidth}
  \centering
  \begin{tikzcd}
                                                        & \mathbb{Z}_2^4.\mathfrak{S}_6                                            &                                                         \\
\mathbb{Z}_2^4.\mathfrak{S}_5 \arrow[ru]                  & \mathbb{Z}_2^4\rtimes\mathfrak{A}_6 \arrow[u]                            & \mathbb{Z}_2^4.\mathfrak{S}_5 \arrow[lu]                  \\
\mathbb{Z}_2^4\rtimes \mathfrak{A}_5 \arrow[ru] \arrow[u] & \mathbb{Z}_2^4.(\mathbb{Z}_5\rtimes\mathbb{Z}_4) \arrow[lu] \arrow[ru] & \mathbb{Z}_2^4\rtimes \mathfrak{A}_5 \arrow[lu] \arrow[u] \\
                                                        & \mathbb{Z}_2^4\rtimes D_{10} \arrow[u] \arrow[lu] \arrow[ru]           &                                                         \\
                                                        & \mathbb{Z}_2^4\rtimes \mathbb{Z}_5 \arrow[u]                           &                                                        
\end{tikzcd}
    \captionof{figure}{13° - 21°}
    \label{fig:13° - 21° Isomorphic}
\end{minipage}%
\end{figure}

\begin{figure}[H]
\centering
\begin{minipage}{.5\textwidth}
  \centering
\begin{tikzcd}
                        & \operatorname{PSp}_4(\mathbb{F}_3)                &                                    &              &                                \\
                        & \mathfrak{S}_6 \arrow[u]             &                                    & \mathfrak{A}_7 &                                \\
\mathfrak{S}_5 \arrow[ru] &                                    & \mathfrak{A}_6 \arrow[ru] \arrow[lu] &              & \operatorname{PSL}_2(\mathbb{F}_7) \arrow[lu] \\
                        & \mathfrak{A}_5 \arrow[lu] \arrow[ru] &                                    &              &                               
\end{tikzcd}
    \captionof{figure}{(A), (C) - (F), (G) y (K)}
    \label{fig:(A), (C) - (F), (G) y (K) Isomorphic}
\end{minipage}%
\begin{minipage}{.5\textwidth}
  \centering
    \begin{tikzcd}
    \mathfrak{S}_5           \\
    \mathfrak{A}_5 \arrow[u]
    \end{tikzcd}
    \captionof{figure}{(B) y (H)}
    \label{fig:(B) y (H) Isomorphic}
\end{minipage}
\end{figure}

\section{Finite Primitive Groups}\label{PrimitiveGroups}

\subsection{Diagram: 1° - 12°}

 For Diagram \ref{fig:1° - 12°} it is enough to describe the group (1°).
\[
\begin{split}
E_2 \star S = 
\frac{\psi}{\sqrt{2}}
\begin{bmatrix}
 i & i & 0 & 0 \\
 1 & -1 & 0 & 0 \\
 0 & 0 & i & i \\
 0 & 0 & 1 & -1 \\
\end{bmatrix},\, 
E_2 \star W_1 =
\begin{bmatrix}
  i & 0 & 0 & 0 \\
 0 & -i & 0 & 0 \\
 0 & 0 & i & 0 \\
 0 & 0 & 0 & -i \\
\end{bmatrix} \\
S \star E_2 =
\frac{\psi}{\sqrt{2}}
\begin{bmatrix}
   i & 0 & i & 0 \\
 0 & i & 0 & i \\
 1 & 0 & -1 & 0 \\
 0 & 1 & 0 & -1 \\
\end{bmatrix}, \,
W_1 \star E_2 =
\begin{bmatrix}
   i & 0 & 0 & 0 \\
 0 & i & 0 & 0 \\
 0 & 0 & -i & 0 \\
 0 & 0 & 0 & -i \\
\end{bmatrix}
\end{split}
\tag{$1^\circ$}
\]

\subsection{Diagram: 13° - 21°}

\

For Diagram \ref{fig:13° - 21°}  it is enough to describe the groups 13°- 17° and 19°. Moreover, we have that (13°) is isomorphic to $\ZZ_2^4 \rtimes \ZZ_5$, (14°) is isomorphic to $\ZZ_2^4 \rtimes D_{10}$ and (19°) is isomorphic to $\ZZ_2^4 . \mathfrak{S}_5$, see \cite{Ivan2019}.
\[
    \begin{split}
    S_1 = 
    \begin{bmatrix}
        0 & 0 & 1 & 0 \\
        0 & 0 & 0 & 1 \\
        1 & 0 & 0 & 0 \\
        0 & 1 & 0 & 0 
    \end{bmatrix}, \,
    S_2 = 
    \begin{bmatrix} 
        0 & 1 & 0 & 0 \\
        1 & 0 & 0 & 0 \\
        0 & 0 & 0 & 1 \\
        0 & 0 & 1 & 0 
    \end{bmatrix}, \,
    T_1=
    \begin{bmatrix}
        1 & 0 & 0 & 0 \\
        0 & 1 & 0 & 0 \\
        0 & 0 & -1 & 0 \\
        0 & 0 & 0 & -1 
    \end{bmatrix}, \\ 
    T_2=
    \begin{bmatrix}
        1 & 0 & 0 & 0 \\
        0 & -1 & 0 & 0 \\
        0 & 0 & 1 & 0 \\
        0 & 0 & 0 & -1 
    \end{bmatrix}, \, 
    T= \frac{\psi}{\sqrt{2}}
    \begin{bmatrix}
    -i & 0 & 0 & i \\
    0 & 1 & 1 & 0 \\
    1 & 0 & 0 & 1 \\
    0 & -i & i & 0 
    \end{bmatrix}
    \end{split}
    \tag{$13^\circ$}
\]

\

\[
(13^\circ) \quad \text{and} \quad
R^2 = 
\begin{bmatrix}
 0 & i & 0 & 0 \\
 i & 0 & 0 & 0 \\
 0 & 0 & -1 & 0 \\
 0 & 0 & 0 & -1 
\end{bmatrix} 
\tag{$14^\circ$}
\]

\

\[
(13^\circ) \quad \text{and} \quad
R = \frac{1}{\sqrt{2}}
\begin{bmatrix}
 1 & i & 0 & 0 \\
 i & 1 & 0 & 0 \\
 0 & 0 & i & 1 \\
 0 & 0 & -1 & -i 
\end{bmatrix} 
\tag{$15^\circ$}
\]

\

\[
(13^\circ) \quad \text{and} \quad
\tag{$16^\circ$}
SB = 
\begin{bmatrix}
 -1 & 0 & 0 & 0 \\
 0 & -1 & 0 & 0 \\
 0 & 0 & i & 0 \\
 0 & 0 & 0 & -i 
\end{bmatrix}
\]

\[
(13^\circ) \quad \text{and} \quad
BR = 
\frac{\psi}{\sqrt{2}}
\begin{bmatrix}
  1 & i & 0 & 0 \\
 i & 1 & 0 & 0 \\
 0 & 0 & i & 1 \\
 0 & 0 & 1 & i 
\end{bmatrix}
\tag{$17^\circ$}
\]

\[
(13^\circ) \quad \text{and} \quad
B = 
\psi
\begin{bmatrix}
 1 & 0 & 0 & 0 \\
 0 & 1 & 0 & 0 \\
 0 & 0 & 1 & 0 \\
 0 & 0 & 0 & -1 \\
\end{bmatrix}
\tag{$19^\circ$}
\]

\subsection{Diagram: (A), (C) - (F), (G), and (K)}

%Del Diagrama \ref{fig:(A), (C) - (F), (G) y (K)} es suficiente describir los grupos $(A)$, $(G)$, $(C)$ y $(E)$, donde los primeros tres son las representaciones $Q_3$, $R_2$, $Q_6$ del paper de Pambianco (2016), mientras que $(E)$ una representación equivalente a $\varphi_{\sqrt{2}}$ del paper de Pambianco (2019). 

For Diagram \ref{fig:(A), (C) - (F), (G) y (K)} it is enough to describe the groups $(A)$, $(G)$, $(C)$ y $(E)$. In addiction, we have that $(A)$ is isomorphic to $\mathfrak{A}_5$, $(G)$ is isomorphic to $\mathfrak{S}_5$, $(C)$ is isomorphic to $\mathfrak{A}_6$, and $(E)$ is isomorphic to $\operatorname{PSL}_2(\mathbb{F}_7)$, see \cite[Chapter VII]{FiniteCollineationGroups}.
\[
F_1 =
\begin{bmatrix}
1 & 0 & 0 & 0 \\
0 & 1 & 0 & 0 \\
0 & 0 & \omega & 0 \\
0 & 0 & 0 & \omega^2 \\
\end{bmatrix}, \,
F_2=
\frac{1}{\sqrt{3}}
\begin{bmatrix}
1 & 0 & 0 & \sqrt{2} \\
0 & -1 & \sqrt{2} & 0 \\
0 & \sqrt{2} & 1 & 0 \\
\sqrt{2} & 0 & 0 & -1 \\
\end{bmatrix}, \,
F_3=
\begin{bmatrix}
\frac{\sqrt{3}}{2} & \frac{1}{2} & 0 & 0 \\[6pt]
\frac{1}{2} & -\frac{\sqrt{3}}{2} & 0 & 0 \\[6pt]
0 & 0 & 0 & 1 \\
0 & 0 & 1 & 0 \\
\end{bmatrix}.
\tag{$A$}
\]

\[
(A) \quad \text{and} \quad
F'' = 
\begin{bmatrix}
0 & 1 & 0 & 0 \\
-1 & 0 & 0 & 0 \\
0 & 0 & 0 & 1 \\
0 & 0 & -1 & 0 \\
\end{bmatrix}
\tag{$G$}
\]

\[
(A) \quad \text{and} \quad
F_4 = 
\begin{bmatrix}
0 & 1 & 0 & 0 \\
1 & 0 & 0 & 0 \\
0 & 0 & 0 & -1 \\
0 & 0 & -1 & 0 \\
\end{bmatrix}
\tag{$C$}
\]

\[
\hat{S} = 
\begin{bmatrix} 
    1 & 0 & 0 & 0\\
    0 & \zeta & 0 & 0 \\
    0 & 0 & \zeta^4 & 0 \\
    0 & 0 & 0 & \zeta^2 
\end{bmatrix},\, 
\hat{T} = 
\begin{bmatrix}
    1 & 0 & 0 & 0 \\
    0 & 0 & 1 & 0 \\
    0 & 0 & 0 & 1 \\
    0 & 1 & 0 & 0 \\
\end{bmatrix}, \,
\hat{R} = \frac{1}{\sqrt{7}}
\begin{bmatrix}
 1 & 1 & 1 & 1 \\
 2 & s_{+} & t_{+} & u_{+} \\
 2 & t_{+} & u_{+} & s_{+}\\
 2 & u_{+} & s_{+} & t_{+}\\
\end{bmatrix}
\tag{$E$}
\]

\subsection{Diagram: (B), and (H)}

%Del Diagrama \ref{fig:(B) y (H)} es necesario describir ambos grupos, $(B)$ y $(H)$, los cuales son las representaciones $Q_2$ y $R_1$, respectivamente, del paper de Pambianco (2016).

For Diagram \ref{fig:(B) y (H)} it is enough to describe both groups, $(B)$ y $(H)$. Likewise, we have that $(B)$ is isomorphic to $\mathfrak{A}_5$ and $(H)$ it is isomorphic to $\mathfrak{S}_5$, see \cite[Chapter VII]{FiniteCollineationGroups}. 
\[
F_1=
\begin{bmatrix}
1 & 0 & 0 & 0 \\
0 & 1 & 0 & 0 \\
0 & 0 & \omega & 0 \\
0 & 0 & 0 & \omega^2 \\
\end{bmatrix},
F_2'=
\begin{bmatrix}
1 & 0 & 0 & 0 \\
0 & \frac{-1}{3} & \frac{2}{3} & \frac{2}{3} \\[6pt]
0 & \frac{2}{3}& \frac{-1}{3} & \frac{2}{3} \\[6pt]
0 & \frac{2}{3} & \frac{2}{3} & \frac{-1}{3} \\
\end{bmatrix}, 
F_3'=
\begin{bmatrix}
-\frac{1}{4} & \frac{\sqrt{15}}{4} & 0 & 0 \\[6pt]
\frac{\sqrt{15}}{4} & \frac{1}{4} & 0 & 0 \\[6pt]
0 & 0 & 0 & 1 \\
0 & 0 & 1 & 0 \\
\end{bmatrix}
\tag{$B$}
\]

\[
(B) \quad \text{and} \quad 
F' = 
\psi
\begin{bmatrix}
1 & 0 & 0 & 0 \\
0& 1& 0 & 0 \\
0 & 0 & 0 & 1 \\
0& 0 & 1 & 0 
\end{bmatrix}.
\tag{$H$}
\]

\section{Some Non-Primitive Finite Groups}\label{Snpg}

%A parte de las representaciones primitivas de los grupos $$\mathbb{A}_5,\mathbb{S}_5,\mathbb{A}_6,\, \mathbb{Z}_2^4 \rtimes \mathbb{Z}_5,\, \mathbb{Z}_2^4 \rtimes D_{10} \, \text{y} \, \operatorname{PSL}_2(\mathbb{F}_7)$$ de la sección anterior, anexamos las restantes representaciones para estos grupos. En particular, $\mathbb{Z}_2^4 \rtimes \mathbb{Z}_5$ y $\mathbb{Z}_2^4 \rtimes D_{10}$ no tienen representaciones aparte de las primitivas. 

Apart from the primitive representations of the groups $$\mathfrak{A}_5,\mathfrak{S}_5,\mathfrak{A}_6,\, \mathbb{Z}_2^4 \rtimes \mathbb{Z}_5,\, \mathbb{Z}_2^4 \rtimes D_{10} \, \text{y} \, \operatorname{PSL}_2(\mathbb{F}_7)$$ given  in the previous section, we present in this section the remaining faithful representations of these groups. In particular, we see that there are no non-primitive representations for the groups  $\mathbb{Z}_2^4 \rtimes \mathbb{Z}_5$, and $\mathbb{Z}_2^4 \rtimes D_{10}$, moreover, this is the case for all groups in the Diagram \ref{fig:13° - 21° Isomorphic}.

% Con esta última frase quiero decir lo que aparece en el Corolario al final de esta sección.

\[
Q_{11} = F_1,\, Q_{12} = F_2'\quad \text{and} \quad Q_{13} = 
\begin{bmatrix}
1 & 0 & 0 & 0 \\
0 & -1 & 0 & 0 \\
0 & 0 & 0 & \lambda_{+} \\
0 & 0 & \lambda_{-} & 0
\end{bmatrix}
\tag{$Q_1$}
\]

\[
Q_{41} =
\begin{bmatrix}
\omega & 0 & 0 & 0 \\
0 & \omega & 0 & 0 \\
0 & 0 & \omega^2 & 0 \\
0 & 0 & 0 & \omega^2
\end{bmatrix},\, 
Q_{42} = 
\begin{bmatrix}
\frac{1}{\sqrt{3}} & 0 & 0 & \frac{\sqrt{2}}{\sqrt{3}} \\[6pt]
0 & \frac{1}{\sqrt{3}} & \frac{\sqrt{2}}{\sqrt{3}} & 0 \\[6pt]
0 & \frac{\sqrt{2}}{\sqrt{3}} & -\frac{1}{\sqrt{3}} & 0 \\[6pt]
\frac{\sqrt{2}}{\sqrt{3}} & 0 & 0 & -\frac{1}{\sqrt{3}} 
\end{bmatrix},\,
Q_{43} =
\begin{bmatrix}
0 & 0 & 0 & \nu_{+} \\
0 & 0 & \nu_{+} \\
0 & \nu_{-} & 0 & 0 \\
\nu_{-} & 0 & 0 & 0
\end{bmatrix}
\tag{$Q_4$}
\]

\[
Q_{51} = Q_{41},\, Q_{52} = Q_{42} \quad \text{and} \quad
Q_{53} =
\begin{bmatrix}
0 & 0 & 0 & \nu_{+} \\
0 & 0 & \nu_{-} & 0 \\
0 & \nu_{+} & 0 & 0 \\
\nu_{-} & 0 & 0 & 0
\end{bmatrix}
\tag{$Q_5$}
\]

\[
(Q_5) \quad \text{and} \quad R_{34} = 
\begin{bmatrix}
0 & 0 & 1 & 0 \\
0 & 0 & 0 & -1 \\
1 & 0 & 0 & 0 \\
0 & -1 & 0 & 0 
\end{bmatrix}
\tag{$R_3$}
\]

\[
(Q_5) \quad \text{and} \quad Q_{74} = 
\frac{1}{\sqrt{5}}
\begin{bmatrix}
0 & 0 & \sqrt{3} & i\sqrt{2} \\
0 & 0 & -i\sqrt{2} & -\sqrt{3} \\
\sqrt{3} & i\sqrt{2} & 0 & 0 \\
-i\sqrt{2} & -\sqrt{3} & 0 & 0 \\
\end{bmatrix}
\tag{$Q_7$}
\]

\[
\tilde{S} = 
\begin{bmatrix} 
    1 & 0 &0 & 0 \\
    0 & \zeta^4& 0 & 0 \\
    0 & 0 & \zeta^2& 0 \\
    0 & 0 & 0& \zeta  \\
\end{bmatrix},\,
\tilde{T}=
\begin{bmatrix}
 1 & 0 & 0 & 0\\
0 & 0 & 1 & 0 \\
0 & 0 & 0 & 1 \\
0 & 1 & 0 & 0 \\
\end{bmatrix},\,
\tilde{R}=
\frac{i}{\sqrt{7}}\begin{bmatrix}
- i \sqrt{7} &0 &0 & 0\\
0& u_{-} & s_{-} & t_{-} \\
0& s_{-} & t_{-}& u_{-} \\
0& t_{-} & u_{-} & s_{-} \\
\end{bmatrix}
\tag{$P$}
\]

Here $(Q_1), (Q_4)$, and $(Q_5)$ are faithful representations of $\mathfrak{A}_5$, $(R_3)$, and $(Q_7)$ are faithful representation of $\mathfrak{S}_5$, and $\mathfrak{A}_6$ respectively. Finally, $(P)$ is the unique non-primitive faithful representation of $\operatorname{PSL}_2(\mathbb{F}_7)$.

By the work of H. Maschke in 1898 \cite{maschke1898bestimmung}, we have that the groups $(A)$, $(B)$, $(Q_1)$, $(Q_4)$, and $(Q_5)$ are all the faithful representations, up to conjugates, of $\mathfrak{A}_5$ in $\PGL_4(\CC)$. The groups $(G)$, $(H)$, and $(R)$ are all faithful representations, up to conjugates, of $\mathfrak{S}_5$ in $\PGL_4(\CC)$. Finally, the groups $(C)$, and $(Q_7)$ are all faithful representations, up to conjugates, of $\mathfrak{A}_6$ in $\PGL_4(\CC)$.

The projective special linear group $\operatorname{PSL}_2(\mathbb{F}_7)$ is the second-smallest non-cyclic simple group. For geometric reasons, it is an important group. For instance, it is the group of automorphisms of the Klein quartic curve,  $K=\{x^3y+z^3x+y^3z=0\}\subseteq\mathbb{P}^2$.
By a famous  Theorem of Hurwitz, see \cite[Chap. III, Theo. 3.9]{miranda1995algebraic}, we know that for smooth curves $C$ of genus $g\geq 2$, the order of its group of automorphism is bounded in terms of the genus $g$. More precisely, the bound is $|\operatorname{Aut}(C)|\leq 84(g-1)$. By the genus formula, the Klein curve $K$ genus is $g=3$, and it is well known that the unique curve of genus $g=3$ attains the Hurwitz bound. In fact, the order of $\operatorname{PSL}_2(\mathbb{F}_7)$ is $168=84(3-1)$.

By \cite[Proposition 4.5, (4)]{faina2019Zp} we have that $(C)$ and $(P)$ are all the faithful representations, up to conjugates, of $\operatorname{PSL}_2(\mathbb{F}_7)$ in $\PGL_4(\CC)$.

\begin{remark}\label{klein}
In \cite{klein1878ueber}, the author classified the unique faithfull representation $\rho$ of $\operatorname{PSL}_2(\mathbb{F}_7)$ in $\operatorname{PGL}_3(\mathbb{C})$. The subgroup $\rho(\operatorname{PSL}_2(\mathbb{F}_7))<\operatorname{PGL}_3(\mathbb{C})$ is generated by the matrices
$$ \begin{bmatrix}
 \zeta^4& 0 & 0 \\
0 & \zeta^2& 0 \\
0 & 0& \zeta  \\
\end{bmatrix}, 
 \begin{bmatrix}
 0& 1 & 0 \\
0 & 0& 1 \\
1 & 0& 0 \\
\end{bmatrix}, 
\text{ and } \, 
\dfrac{i}{\sqrt{7}}\begin{bmatrix}
u_{-}& s_{-} & t_{-} \\
s_{-} & t_{-}& u_{-} \\
t_{-} &  u_{-} & s_{-} \\
\end{bmatrix}.$$
\
\end{remark}

Note that the projective representation $(P)$ of $\operatorname{PSL}_2(\mathbb{F}_7)$ is a trivial extension of the last representation, found by Klein.

By \cite[Corollary 5.2.]{Ivan2019}, we have that (13°) is the unique faithful representation, up to conjugates, of $\ZZ_2^4 \rtimes \ZZ_5$ in $\PGL_4(\CC)$. The proof of this corollary shows there is no non-primitive subgroup of $\PGL_4(\CC)$ isomorphic to $\ZZ_2^4 \rtimes \ZZ_5$. Now we are in conditions to show that (14°) is the unique faithful representation, up to conjugates, of the group $\ZZ_2^4 \rtimes D_{10}$ in $\PGL_4(\CC)$. Suppose  that $\overline{G}_{160}$ is a subgroup of $\PGL_4(\CC)$ isomorphic to $\ZZ_2^4 \rtimes D_{10}$, as $\ZZ_2^4 \rtimes \ZZ_5$ is a subgroup of $\ZZ_2^4 \rtimes D_{10}$, then we have that $\overline{G}_{80}$ is a subgroup of $\overline{G}_{160}$ and $\PGL_4(\CC)$ isomorphic to $\ZZ_2^4 \rtimes \ZZ_5$. By the previous considerations $\overline{G}_{80}$ is a primitive subgroup of $\PGL_4(\CC)$, thus $\overline{G}_{160}$ is also a primitive subgroup of $\PGL_4(\CC)$. By Blichfeltd's classification of all finite primitive subgroups of $\operatorname{PGL}_4(\mathbb{C})$, we get that $\overline{G}_{160}$ is conjugate to (14°),  note that there is an only one finite primitive group in the Blichfeltd's classification of order 160. This argument works for all groups in the diagram \ref{fig:13° - 21° Isomorphic}. In conclusion, from the discussion presented in the last paragraph, we obtain  Corollary \ref{c}.

%\begin{corollary}\label{c}
%All subgroups of $\PGL_4(\CC)$ isomorphic to $\ZZ_2^4 \rtimes D_{10}$ are conjugates of (14°). 
%\end{corollary}

\begin{corollary}\label{c}
Let $G$ be subgroup of $\PGL_4(\CC)$ isomorphic to some group of the Diagram \ref{fig:13° - 21° Isomorphic}. Then $G$ is conjugate to some group of the Diagram \ref{fig:13° - 21°}. 
\end{corollary}

\section{Invariant Subspaces of Quartics}\label{QuarticInvariants}

%\begin{remark}
%Las formas homogéneas señaladas con el símbolo $(*)$ son singulares y las señaladas con el símbolo $(\blacklozenge)$ son no-singulares.
%\end{remark}

In this section, we find for each group $G$ in the previous two sections the collection of all (maximal) $G$-invariant subspace of quartics. Since all these linear subspaces are finite-dimensional, then it is enough to show a finite basis by each of such subspaces. For each one of these basis $\{p_1,p_2,\ldots,p_r \}$ we show a criterion over the scalar $\lambda_1,\lambda_2,\ldots,\lambda_r$ to determinate which linear combinations $\lambda_1 p_1 + \lambda_2 + \cdots + \lambda_r p_r$ are non-singular. In particular, if the basis consists of only one polynomial $p$,  we will tag this polynomial with $(*)$ to mean \textit{$p$ is singular}, by the contrary, if $p$ is non-singular, then we will tag $p$ with $(\blacklozenge)$. For example, the polynomial $x_1^2 x_2^2+x_0^2 x_3^2 -2x_0 x_1 x_3 x_2$ in \ref{PrimerSub1°} is singular, but the polynomial $2 x_0^4+ 6 x_1 x_2 x_3 x_0+x_1 x_3^3+x_1^3 x_2+x_2^3 x_3$ in \ref{Non-singularPSL(2,7)} is non-singular.

\textbf{Respect to Diagram \ref{fig:1° - 12°}.} 

\subsection{Invariant subspaces of quartics by (1°)}\label{QuarticsInvariantby(1°)}

\subsubsection{First subspace  $I_1$}\label{PrimerSub1°}

\[
x_1^2 x_2^2+x_0^2 x_3^2 -2x_0 x_1 x_3 x_2. \tag{*}
\]

\subsubsection{Second subspace $\operatorname{I}_2$}\label{SegundoSub1°}

\[
x_0^4+x_1^4+x_2^4+x_3^4 -8 x_0 x_1 x_2 x_3 -2 \left(x_1^2 x_2^2+x_0^2 x_3^2\right)+2 i \sqrt{3} \left(x_0^2 x_1^2+x_3^2 x_1^2+x_0^2 x_2^2+x_2^2 x_3^2\right). 
\tag{*}
\]

\subsubsection{Third subspace $\operatorname{I}_3$}

\[
x_0^4+x_1^4+x_2^4+x_3^4+8 x_0 x_1 x_2 x_3 +2 \left(x_1^2 x_2^2+x_0^2 x_3^2\right)+2 i \sqrt{3} \left(x_0^2 x_1^2-x_3^2 x_1^2-x_0^2 x_2^2+x_2^2 x_3^2\right). \tag{*}
\]

\subsubsection{Fourth subspace $\operatorname{I}_4$}

\[
x_0^4+x_1^4+x_2^4+x_3^4 + 8 x_1 x_2 x_3 x_0 + 2 \left(x_1^2 x_2^2+x_0^2 x_3^2\right)-2 i \sqrt{3} \left(x_0^2 x_1^2-x_3^2 x_1^2-x_0^2 x_2^2+x_2^2 x_3^2\right). \tag{*}
\]

\subsubsection{Fifth subspace $\operatorname{I}_5$}\label{QuintoSub1°}

\[
x_0^4+x_1^4+x_2^4+x_3^4 -8 x_1 x_2 x_3 x_0 -2 \left(x_1^2 x_2^2+x_0^2 x_3^2\right)-2 i \sqrt{3} \left(x_0^2 x_1^2+x_3^2 x_1^2+x_0^2 x_2^2+x_2^2 x_3^2\right). \tag{*}
\]

For the sake of clarity, we want to mention that it is unnecessary to compute invariant quartics for the remaining groups on Diagram \ref{fig:1° - 12°}. Any other group in Diagram \ref{fig:1° - 12°} contains (1°), and the invariant quartics by (1°) are singular (*). Therefore, the other groups in this diagram are done.

%\subsection{Invariant subspaces of quartics by (2°)} 

%\subsubsection{Unique subspace of (2°)} The subspace \ref{PrimerSub1°}.

%\subsection{Invariant subspaces of quartics by (3°)}

%\subsubsection{Unique subspace of (3°)} The subspace \ref{PrimerSub1°}.

%\subsection{Invariant subspaces of quartics by (4°)}

%\subsubsection{Unique subspace of (4°)} The subspace \ref{PrimerSub1°}.

%\subsection{Invariant subspaces of quartics by (10°)}

%\subsubsection{First subspace of (10°)} The subspace \ref{PrimerSub1°}.

%\subsubsection{Second subspace of (10°)} The subspace \ref{SegundoSub1°}.

%\subsubsection{Third subspace of (10°)} The subspace \ref{QuintoSub1°}.

%\

\textbf{Respect to Diagram \ref{fig:13° - 21°}.} 

\subsection{Invariant subspaces of quartics by (13°)}

\subsubsection{First subspace $\operatorname{XIII}_1$}\label{PrimerSub13°}

%Este es f_4
\[
k = x_0^4+x_1^4+x_2^4+x_3^4+6 \left(x_0^2 x_1^2-x_2^2 x_1^2+x_3^2 x_1^2+x_0^2 x_2^2-x_0^2 x_3^2+x_2^2 x_3^2\right). 
\tag{$\blacklozenge$}
\]

\subsubsection{Second subspace $\operatorname{XIII}_2$}
% Este es q_1
\[
\begin{split}
    &-\frac{-1+4\mu-\mu^2-2\mu^3+2\mu^4}{\mu_5^2-1}(x_0^4+x_1^4+x_2^4+x_3^4)-6\frac{\mu^2-\mu^3-\mu^4}{1+\mu}(x_0^2x_2^2+x_1^2x_3^2) \\ &+6\frac{1+\mu^2}{\mu^2-1}(x_0^2x_1^2+x_2^2x_3^2)-6\frac{1-\mu^2+2\mu^4}{\mu^2-1}(x_0^2x_3^2+x_1^2x_2^2)+24x_0x_1x_2x_3. 
\end{split}
\tag{*}
\]

\subsubsection{Third subspace $\operatorname{XIII}_3$}

%Este es q_2
\[
\begin{split}
    &(1+2\mu+4\mu^2+2\mu^3+\mu^4)(x_0^4+x_1^4+x_2^4+x_3^4) + 6(1+\mu^4)(x_0^2x_1^2+x_2^2x_3^2) \\ &-6(1-2\mu^3-\mu^4)(x_0^2x_3^2+x_1^2x_2^2)+6(\mu^4-2\mu-1)(x_0^2x_2^2+x_1^2x_3^2)+24(\mu^4-1)x_0x_1x_2x_3. 
\end{split}
\tag{*}
\]

\subsubsection{Fourth subspace $\operatorname{XIII}_4$}

%Este es q_3
\[
\begin{split}
     &-(1+2\mu+4\mu^2+2\mu^3+\mu^4)(x_0^4+x_1^4+x_2^4+x_3^4)-6(1+\mu^4)(x_0^2x_1^2+x_2^2x_3^2) \\ &-6(1-2\mu^3-\mu^4)(x_0^2x_2^2+x_1^2x_3^2)-6(1+2\mu-\mu^4)(x_0^2x_3^2+x_1^2x_2^2)+24(-1+\mu^4)x_0x_1x_2x_3. 
\end{split}
\tag{*}
\]

\subsubsection{Fifth subspace $\operatorname{XIII}_5$}

%Este es q_4
\[
\begin{split}
    &(-1+4\mu-\mu^2-2\mu^3+2\mu^4)(x_0^4+x_1^4+x_2^4+x_3^4) -6(1+\mu^2)(x_0^2x_1^2+x_2^2x_3^2) \\ &-6(1-\mu^2+2\mu^4)(x_0^2x_2^2+x_1^2x_3^2)-6(1-\mu^2+2\mu^3)(x_0^2x_3^2+x_1^2x_2^2) + 24(-1+\mu^2)x_0x_1x_2x_3. 
\end{split}
\tag{*}
\]

\subsection{Invariant subspaces of quartics by (14°)}

\subsubsection{$\operatorname{XIV}_1$} It is the subspace \ref{PrimerSub13°}.  

\subsection{Invariant subspaces of quartics by (15°)}

\subsubsection{$\operatorname{XV}_1$} It is the subspace \ref{PrimerSub13°}.  

\subsection{Invariant subspaces of quartics by (16°)}\label{QuarticInvarianby(16°)} There is no non-trivial subspace of invariant quartics by this group.  The last statement implies that there is also no non-trivial subspace of invariant quartics by groups (18°), (20°), and (21°).
\subsection{Invariant subspaces of quartics by (17°)}

\subsubsection{$\operatorname{XVII}_1$}  It is the subspace \ref{PrimerSub13°}.

\subsection{Invariant subspaces of quartics by (19°)}\label{QuarticInvariantby(19°)}
 
\subsubsection{$\operatorname{XIX_1}$} It is the subspace \ref{PrimerSub13°}.

\textbf{
Respect to Diagram \ref{fig:(A), (C) - (F), (G) y (K)}.}

\subsection{Invariant subspaces of quartics by (A)} 

\subsubsection{$\operatorname{A}_1$}\label{ÚnicoSub(A)}

\[
\begin{split}
    h_0 & = \frac{x_0^4}{2 \sqrt{3}}+\frac{1}{3} x_1 x_0^3+x_2 x_3 x_0^2+\frac{1}{3} x_1^3 x_0+\frac{1}{3} \sqrt{2} x_2^3 x_0+\sqrt{\frac{2}{3}} x_3^3 x_0 \\ & +\sqrt{\frac{2}{3}} x_1 x_2^3-\frac{1}{3} \sqrt{2} x_1 x_3^3+x_1^2 x_2 x_3-\frac{x_1^4}{2 \sqrt{3}}, \\
    h_1 & = \frac{x_0^4}{6}+\frac{x_1 x_0^3}{\sqrt{3}}-\frac{1}{2} x_1^2 x_0^2+\sqrt{3} x_2 x_3 x_0^2+\frac{1}{3} \sqrt{2} x_3^3 x_0+x_1 x_2 x_3 x_0 \\ & +\frac{2}{3} \sqrt{2} x_1 x_2^3-\sqrt{\frac{2}{3}} x_1 x_3^3-\frac{1}{2} x_2^2 x_3^2-\frac{x_1^4}{3}.
\end{split}
\]

Denote by  $D(\lambda_0,\lambda_1)$ the discriminant  of the linear combination $\lambda_0 h_0 + \lambda_1 h_1$, it is a homogeneous polynomial in the parameters $\lambda_0,\lambda_1$ of degree $(3+1) \cdot (4-1)^3 = 108$. Since the zero locus of $D(\lambda_0,\lambda_1)$ lies in $\PP^1$, then there are finite many non-trivial solutions of $D(\lambda_0,\lambda_1) = 0$. Moreover, the number of solutions is bounded by $108$, the degree of $D(\lambda_0,\lambda_1)$. Some of the points $[\lambda_0 : \lambda_1] \in \PP^1$ such that $\{\lambda_0 h_0 + \lambda_1 h_1 = 0 \}$ is a singular invariant surface by $(A)$, these points $[\lambda_0 : \lambda_1]$ are the following:
\[
\begin{split}
    \left[ - 2 \sqrt{2} : \sqrt{i \sqrt{15} -1} \right],\, \left[ - 2 \sqrt{2} : \sqrt{- i \sqrt{15} -1} \right], \, [1 : - \sqrt{3}], \,  \left[ 47 :8 \left(1-4 \sqrt{3}\right) \right]  \\
    \left[ 47 : - 8 \left(1 + 4 \sqrt{3}\right) \right], \, \left[ 4 : -\sqrt{3}+i \sqrt{5} \right], \; \text{and} \; \left[ 4 : -\sqrt{3} - i \sqrt{5} \right]. 
\end{split}
\]
We found these points numerically with the help of the software \texttt{Mathematica}. 
Unfortunately, to find all zeros of $D(\lambda_0,\lambda_1)$ numerically, or to compute the discriminant $D(\lambda_0,\lambda_1)$ by Macaulay's method is computationally heavy. %We recommend to the interest reader to see the work of G. Staglianò in \cite{Staglianò18} for a more efficient algorithms in this matter, which he implemented in the software \texttt{Macaulay2}. 

\subsection{Invariant subspaces of quartics by (G)} 

\subsubsection{$\operatorname{G}_1$} 

\[
\begin{split}
    f_0 & = -\frac{x_0^4}{12}+\frac{x_1 x_0^3}{2 \sqrt{3}}-\frac{1}{2} x_1^2 x_0^2+\frac{1}{2} \sqrt{3} x_2 x_3 x_0^2+x_1 x_2 x_3 x_0-\frac{x_3^3 x_0}{3 \sqrt{2}}-\frac{x_1^3 x_0}{2 \sqrt{3}}-\frac{x_2^3 x_0}{\sqrt{6}} \\ & +\frac{x_1 x_2^3}{3 \sqrt{2}}-\frac{1}{2} x_2^2 x_3^2-\frac{1}{2} \sqrt{3} x_1^2 x_2 x_3-\frac{x_1^4}{12}-\frac{x_1 x_3^3}{\sqrt{6}}.
\end{split}
\tag{$\blacklozenge$}
\]

\subsubsection{$\operatorname{G}_2$}

\[
\begin{split}
    f_1 & = \frac{x_0^4}{2 \sqrt{3}}+\frac{1}{3} x_1 x_0^3+x_2 x_3 x_0^2+\frac{1}{3} x_1^3 x_0+\frac{1}{3} \sqrt{2} x_2^3 x_0+\sqrt{\frac{2}{3}} x_3^3x_0 \\ & +\sqrt{\frac{2}{3}} x_1 x_2^3-\frac{1}{3} \sqrt{2} x_1 x_3^3+x_1^2 x_2 x_3-\frac{x_1^4}{2 \sqrt{3}}.
\end{split}
\tag{$\blacklozenge$}
\]

\subsection{Invariant subspaces of quartics by (C)} There is no  a non-trivial subspace of invariant quartics  by this group.

\subsection{Invariant subspaces of quartics by (E)}

\subsubsection{$\operatorname{E}_1$}\label{Non-singularPSL(2,7)}

%Esta es de PSL(2,7), no-singular
\[
q = 2 x_0^4+ 6 x_1 x_2 x_3 x_0+x_1 x_3^3+x_1^3 x_2+x_2^3 x_3. 
\tag{$\blacklozenge$}
\]

\

\textbf{
Respect to Diagram \ref{fig:(B) y (H)}.}

\subsection{Invariant subspaces of quartics by (B)}

\subsubsection{$\operatorname{B}_1$}\label{UnicoSub(B)}

\[
\begin{split}
    g_0 &= \frac{x_0^4}{\sqrt{15}}-\frac{x_2 x_3 x_0^2}{\sqrt{15}}-\frac{x_1^2 x_0^2}{2 \sqrt{15}}-\frac{1}{3} x_1^3 x_0-\frac{1}{3} x_2^3 x_0-\frac{1}{3} x_3^3 x_0+x_1 x_2 x_3 x_0 \\ &+\frac{1}{6} \sqrt{\frac{5}{3}} x_1^4-\frac{1}{3} \sqrt{\frac{5}{3}} x_1 x_2^3-\frac{1}{3} \sqrt{\frac{5}{3}} x_1 x_3^3+\frac{1}{2} \sqrt{\frac{5}{3}} x_2^2 x_3^2, \\
    g_1 &= \frac{x_0^4}{4}+\frac{1}{2} x_1^2 x_0^2+x_2 x_3 x_0^2+\frac{x_1^4}{4}+x_2^2 x_3^2+x_1^2 x_2 x_3.
\end{split}
\]

The linear combination $\lambda_0 g_0 + \lambda_1 g_1$ is non-singular if and only if 
\begin{center}
    $-20 \lambda_1 \lambda_0^4+107 \sqrt{15} \lambda_1^2 \lambda_0^3+1140 \lambda_1^3 \lambda_0^2+180 \sqrt{15} \lambda_1^4 \lambda_0 \neq 0$
\end{center}

%\textcolor{red}{He encontrado el criterio de no singularidad para este subespacio. Al final use otro método diferente al discriminante a la Macaulay, el cual funciono bien en este caso. Desafortunadamente no funciona igual de bien para el restante subespacio (el grupo (A)). Como comentario adicional, leyendo con detalle los papers de Pambianco (2016) y (2019), vemos que elllos no encuentran los criterios para estos dos subespacios (A) y (B).}

\subsection{Invariant subspaces of quartics by (H)}

\subsubsection{$\operatorname{H}_1$} It is the subspace \ref{UnicoSub(B)}.

\vspace{5mm}

\textbf{The non-primitive groups.}

\subsection{Invariant subspaces of quartics by (Q\textsubscript{1})}

\subsubsection{$\operatorname{Q}_{11}$}
\[
\begin{split}
    & x_1^4+4x_2 x_3 x_1^2+4x_2^2 x_3^2,\, x_1^2 x_0^2+2x_2 x_3 x_0^2, \\ 
    & x_0^4.
\end{split}
\tag{*}
\]
Let's note that all linear combinations of these three polynomials are singular, since each one of these has degree  less than $2$ in the variable $x_2$. Then each polynomial in this subspace has degree less than $2$ in the variable $x_2$ and by Lemma \ref{MSC} we get that all polynomials in this subspace are singulars. 

\subsection{Invariant subspaces of quartics by (Q\textsubscript{4})}

\subsubsection{$\operatorname{Q}_{41}$}
\[
    -\frac{1}{2} x_0^2 x_2^2+x_0 x_1 x_3 x_2-\frac{1}{2} x_1^2 x_3^2. \tag{*}
\]

\subsection{Invariant subspaces of quartics by (Q\textsubscript{5})} There is no   non-trivial subspace of  invariant quartics by this group.

\subsection{Invariant subspaces of quartics by (R\textsubscript{3})} There is no   non-trivial subspace of  invariant quartics by this group.

\subsection{Invariant subspaces of quartics by (Q\textsubscript{7})} There is no   non-trivial subspace of  invariant quartics by this group.

\subsection{Invariant subspaces of quartics by (P)}

\subsubsection{$\operatorname{P}_1$}\label{P}
\[
\begin{split}
    p_0 & = x_0^4, \\
    p_1 & = x_2 x_1^3+x_3^3 x_1+x_2^3 x_3.
\end{split}
\]

The linear combination $\lambda_0 p_0 + \lambda_1 p_1$ is non-singular if and only if $\lambda_0 \lambda_1 \neq 0$. Certainly, by  Lemma \ref{MSC} it is necessary that $\lambda_1$ and $\lambda_2$ are non-zero. Now, if $b \in \PP^3$ is a singular point of $\lambda_0 p_0 + \lambda_1 p_1$, then we find a singular point $\hat{b} \in \PP^2$  of $p_1$, but $\{ p_1 = 0 \}$ is the smooth Klein quartic curve, contradiction. Therefore, $\lambda_0$ and $\lambda_1$ are non-zero is a criterion of non-singularity of $\lambda_0 p_0 + \lambda_1 p_1$, this criterion is equivalent to $\lambda_0 \lambda_1 \neq 0$.

\begin{remark}
 All computations have been made with the help of the software \texttt{Mathematica}. %Although it has some limitations in the computations of discriminants, see the comments in \ref{ÚnicoSub(A)}. 
\end{remark}

\section{Results}\label{Results}

Here we present the results obtained in the last two sections. It means, we lists all the smooth quartic surfaces by the finite primitive groups of $\operatorname{PGL}_4(\mathbb{C})$.

\begin{theorem}[Diagram \ref{fig:1° - 12°}]

There are no smooth quartic surfaces invariant by at least one primitive subgroup listed in Diagram \ref{fig:1° - 12°}.
\end{theorem}

\begin{proof}
See Section \ref{QuarticInvariants}, subsection \ref{PrimerSub1°}.
\end{proof}

%\begin{theorem}
%None groups in the Diagram \ref{fig:1° - 12°} fix any smooth surface in $\PP^3$ of degree $4$.
%\end{theorem}
%\begin{proof}
%If $G$ is a group in this diagram and $X$ is a $G$-invariant smooth surface of degree $4$, then $X$ is (1°)-invariant. If $X = \{ f = 0 \}$, then $f$ is (1°)-invariant non-singular quartic form, this is a  contradiction with \ref{QuarticsInvariantby(1°)} since there is no exists a non-singular quartic homogeneous form invariant by (1°). 
%\end{proof}

\begin{theorem}[Diagram \ref{fig:13° - 21°}]

The unique smooth quartic surface invariant by at least one primitive group on Diagram \ref{fig:13° - 21°} is given by the polynomial $$ k=x_0^4+x_1^4+x_2^4+x_3^4+6 \left(x_0^2 x_1^2-x_2^2 x_1^2+x_3^2 x_1^2+x_0^2 x_2^2-x_0^2 x_3^2+x_2^2 x_3^2\right).$$ Moreover, such surface is invariant by the group (19°), which is isomorphic to $\ZZ_2^4.\mathbb{S}_5$, and by its subgroups (13°), (14°), (15°), (17°), and (19°). On the other hand, the primitive groups (16°), (18°), (20°), and (21°) do not fix any smooth quartic surface. 

\end{theorem}

%\begin{theorem}\label{ConclusionDiagrama13°}
%Consider the subgroups of $\PGL_4(\CC)$ in the Diagram \ref{fig:13° - 21°}. Then the group $(19^\circ)$ isomorphic to $\ZZ_2^4.\mathbb{S}_5$ and its subgroups in this diagram are the only groups in this diagram which let invariant some smooth quartic surface. In fact, there is only one smooth quartic surface invariant by each one of these groups, i.e., $\{ k = 0 \}$, where 
%\[
%k = x_0^4+x_1^4+x_2^4+x_3^4+6 \left(x_0^2 x_1^2-x_2^2 x_1^2+x_3^2 x_1^2+x_0^2 x_2^2-x_0^2 x_3^2+x_2^2 x_3^2\right). 
%\]
%\end{theorem}
\begin{proof}
By Section \ref{QuarticInvariants} it is enough to prove that none of groups (16°), (18°), (20°), or (21°) fixes any smooth surface of degree $4$. If $X=\{ f = 0\}$ is a smooth quartic surface invariant by some of these groups, then by Diagram \ref{fig:13° - 21°} $X$ must be invariant by (16°), thus $f$ is a non-singular quartic form invariant by (16°), contradiction with \ref{QuarticInvarianby(16°)}.
\end{proof}

%\begin{theorem}
%Consider the subgroups of $\PGL_4(\CC)$ in the Diagram \ref{fig:(A), (C) - (F), (G) y (K)}. Then $(C)$, $(K)$, $(D)$ and $(F)$ do not fix any smooth quartic surface. By other hand, the remaining groups of this diagrams let fix some smooth quartic surface.
%\end{theorem}

\begin{theorem}[Diagram \ref{fig:(A), (C) - (F), (G) y (K)}]
The smooth quartic surfaces fixed by primitives groups on Diagram \ref{fig:(A), (C) - (F), (G) y (K)} are the following: 

\begin{enumerate}
    \item[(A)] All the (A)-invariant quartic surfaces are in the pencil of quartics
 \[
\begin{split}
    \Big\{\lambda_0\Big( &  \frac{x_0^4}{2 \sqrt{3}}+\frac{1}{3} x_1 x_0^3+x_2 x_3 x_0^2+\frac{1}{3} x_1^3 x_0+\frac{1}{3} \sqrt{2} x_2^3 x_0+\sqrt{\frac{2}{3}} x_3^3 x_0 \\ & +\sqrt{\frac{2}{3}} x_1 x_2^3-\frac{1}{3} \sqrt{2} x_1 x_3^3+x_1^2 x_2 x_3-\frac{x_1^4}{2 \sqrt{3}}\Big)+ \\
    \lambda_1\Big(  & \frac{x_0^4}{6}+\frac{x_1 x_0^3}{\sqrt{3}}-\frac{1}{2} x_1^2 x_0^2+\sqrt{3} x_2 x_3 x_0^2+\frac{1}{3} \sqrt{2} x_3^3 x_0+x_1 x_2 x_3 x_0 \\ & +\frac{2}{3} \sqrt{2} x_1 x_2^3-\sqrt{\frac{2}{3}} x_1 x_3^3-\frac{1}{2} x_2^2 x_3^2-\frac{x_1^4}{3}\Big)=0 \Big\}.
\end{split}
\]
 %where the singular ones are the fibers \textcolor{red}{ Es lo que nos falta}
  where the singular  the fibers are  finitely many and at least  must include the fibers over the following points: 
 \[
\begin{split}
    \left[ - 2 \sqrt{2} : \sqrt{i \sqrt{15} -1} \right],\, \left[ - 2 \sqrt{2} : \sqrt{- i \sqrt{15} -1} \right], \, [1 : - \sqrt{3}], \,  \left[ 47 :8 \left(1-4 \sqrt{3}\right) \right]  \\
    \left[ 47 : - 8 \left(1 + 4 \sqrt{3}\right) \right], \, \left[ 4 : -\sqrt{3}+i \sqrt{5} \right], \; \text{and} \; \left[ 4 : -\sqrt{3} - i \sqrt{5} \right]. 
\end{split}
\]

%\textcolor{red}{He anexado los puntos malos que aparecen tanto en la introdicción como en \ref{ÚnicoSub(A)}.}
 
 \item[(G)] There are only two smooth quartic surfaces invariant by (G), they are: 
 \[
\begin{split}
   \Big\{  &-\frac{x_0^4}{12}+\frac{x_1 x_0^3}{2 \sqrt{3}}-\frac{1}{2} x_1^2 x_0^2+\frac{1}{2} \sqrt{3} x_2 x_3 x_0^2+x_1 x_2 x_3 x_0-\frac{x_3^3 x_0}{3 \sqrt{2}}-\frac{x_1^3 x_0}{2 \sqrt{3}}-\frac{x_2^3 x_0}{\sqrt{6}} \\ & +\frac{x_1 x_2^3}{3 \sqrt{2}}-\frac{1}{2} x_2^2 x_3^2-\frac{1}{2} \sqrt{3} x_1^2 x_2 x_3-\frac{x_1^4}{12}-\frac{x_1 x_3^3}{\sqrt{6}}=0 \Big\},
\end{split}
\]
and 

\[
\begin{split}
    \Big\{ & \frac{x_0^4}{2 \sqrt{3}}+\frac{1}{3} x_1 x_0^3+x_2 x_3 x_0^2+\frac{1}{3} x_1^3 x_0+\frac{1}{3} \sqrt{2} x_2^3 x_0+\sqrt{\frac{2}{3}} x_3^3x_0 \\ & +\sqrt{\frac{2}{3}} x_1 x_2^3-\frac{1}{3} \sqrt{2} x_1 x_3^3+x_1^2 x_2 x_3-\frac{x_1^4}{2 \sqrt{3}}=0 \Big\}.
\end{split}
\]

\item[(E)] The unique smooth quartic invariant surface by the group (E) is: 
$$\{2 x_0^4+ 6 x_1 x_2 x_3 x_0+x_1 x_3^3+x_1^3 x_2+x_2^3 x_3=0\}.$$

\end{enumerate}

In addition, the primitive groups (C), (D), (K), and (F) do not fix any smooth quartic surface.
\end{theorem}

\begin{proof}
It follows immediately by Diagram \ref{fig:(A), (C) - (F), (G) y (K)} and Section \ref{QuarticInvariants}.
\end{proof}

\begin{theorem}[Diagram \ref{fig:(B) y (H)}]
The  quartic surfaces invariant by primitives groups (B), and (H) are given by the pencil of quartic surfaces

\[
\begin{split}
   \Big\{\lambda_0&( \frac{x_0^4}{\sqrt{15}}-\frac{x_2 x_3 x_0^2}{\sqrt{15}}-\frac{x_1^2 x_0^2}{2 \sqrt{15}}-\frac{1}{3} x_1^3 x_0-\frac{1}{3} x_2^3 x_0-\frac{1}{3} x_3^3 x_0+x_1 x_2 x_3 x_0 \\ &+\frac{1}{6} \sqrt{\frac{5}{3}} x_1^4-\frac{1}{3} \sqrt{\frac{5}{3}} x_1 x_2^3-\frac{1}{3} \sqrt{\frac{5}{3}} x_1 x_3^3+\frac{1}{2} \sqrt{\frac{5}{3}} x_2^2 x_3^2)+ \\
   \lambda_1 & (\frac{x_0^4}{4}+\frac{1}{2} x_1^2 x_0^2+x_2 x_3 x_0^2+\frac{x_1^4}{4}+x_2^2 x_3^2+x_1^2 x_2 x_3)=0 \Big\},
\end{split}
\]

%where the singular ones are the fibers on the eight points of $\mathbb{P}^{1}_{\lambda_0,\lambda_1}$
%$$[0:1],[6\sqrt{15}:1],[-\sqrt{15}:4],[-2\sqrt{15}:5],[15:-4\sqrt{15}],[6:-\sqrt{15}],[90:\sqrt{15}], \text{ and } [1:0].$$
where the singular ones are the fibers on the five points of $\mathbb{P}^{1}_{\lambda_0,\lambda_1}$
$$[0:1],\,[1:0],\, [\sqrt{15} : -4],\, [2 \sqrt{15} : -5], \, \text{and} \, [6 \sqrt{15} : 1] .$$
\end{theorem}

\begin{proof}
It is a straightforward  consequence by the results presented in Section \ref{QuarticInvariants}. 
\end{proof}

\begin{corollary}
The projective automorphism group of the smooth quartic surface $\{ k = 0 \}$ is the group $(19^\circ)$. 
%generate by the collineations
%\[
%S_1,S_2,T_1,T_2,T, \quad \text{and} \quad 
%B = \psi
%\begin{bmatrix}
%1 & 0 & 0 & 0 \\
%0 & 1 & 0 & 0 \\
%0 & 0 & 1 & 0 \\
%0 & 0 & 0 & -1 
%\end{bmatrix}.
%\]
Moreover, $\operatorname{PAut}\{ k = 0 \} \cong \mathbb{Z}_2^4.\mathfrak{S}_5$, and $\{k = 0\}$ 
%$$\{ k = x_0^4+x_1^4+x_2^4+x_3^4+6 \left(x_0^2 x_1^2-x_2^2 x_1^2+x_3^2 x_1^2+x_0^2 x_2^2-x_0^2 x_3^2+x_2^2 x_3^2\right)= 0\}$$ 
is projectively equivalent to  $$\{h_{12} = x_0^4 + x_1^4 + x_2^4 + x_3^4 + 12 x_0 x_1 x_2 x_3=0\}.$$

Note that, $h_{12}$ is the quartic form which Burnside conjectured to be the most symmetric smooth quartic surface. 
\end{corollary}

\begin{proof}
By \ref{QuarticInvariantby(19°)}, we have that $\operatorname{PAut}\{ k = 0 \} = \operatorname{PAut}(k) \supset (19^\circ)$. This implies $\operatorname{PAut}(k)$ is a primitive group of $\PGL_4(\CC)$, which it is finite by \cite[Theorem 1]{matsumura1963automorphisms}. By Blichfeltd's classification of finite primitive groups of $\PGL_4(\CC)$ we have $\operatorname{PAut}(k)$ is conjugate to some unique group in the diagrams of Section \ref{Diagrams}, namely $G$. The order of $G$ is $|\operatorname{PAut}(k)| \geq |(19^\circ)| = 1920$ and by Section \ref{QuarticInvariants} we have (19°) is the group of maximum order between the groups of diagrams in Section \ref{Diagrams} which fix some smooth quartic surface, thus $|\operatorname{PAut}(k)|=|G| = 1920$. Therefore $\operatorname{PAut}(k) \supset (19^\circ)$ have the same order and they must to be equal. Therefore, $\operatorname{PAut}\{ k = 0 \}=\operatorname{PAut}(k) = (19^\circ)$.

Now we are going to prove that $\{ k = 0 \}$ and $\{ h_{12} = 0 \}$ are projectively equivalent. By \cite[Section 5]{faina2016s5} we have $h_{12}$ is non-singular, $|\operatorname{PAut}\{ h_{12} = 0 \}| = |\operatorname{PAut}(h_{12})|= 1920$ and $\operatorname{PAut}(h_{12})$ has a subgroup conjugate to the primitive group $(A)$. Thus, $\operatorname{PAut}(h_{12})$ is a primitive finite group of $\PGL_4(\CC)$ which is conjugate of $(19^\circ)=\operatorname{PAut}(k)$ by Blichfeldt's classification. Lets take $(T) \in \PGL_4(\CC)$ such that $\operatorname{PAut}(k)(T) = (T)\operatorname{PAut}(h_{12})$. It is clear $(T)\{ h_{12} = 0 \}$ is invariant by $\operatorname{PAut}(k) = (19^\circ)$, and by \ref{QuarticInvariantby(19°)} must be the case $(T)\{h_{12} = 0 \}$ and $\{k = 0 \}$ are equals. Therefore, $\{h_{12} = 0 \}$ and $\{ k = 0\}$ are projective equivalent. 
\end{proof}
By the previous  result, we retrieve a noticeable outcome. In fact, we get  that the surface $\{ h_{12} = 0 \}$ is the most symmetric quartic smooth surface, up to projective equivalence. The previous proof also proves Corollary B in the introduction. For completeness we present the smooth quartic invariant surfaces by the non-primitive representations of the groups in Diagram \ref{fig:(A), (C) - (F), (G) y (K)}.

\begin{theorem}[Non-Primitive Groups in Diagram \ref{fig:(A), (C) - (F), (G) y (K)}]

The group $\operatorname{PSL}_2(\mathbb{F}_7)$ seen in $\operatorname{PGL}_4(\mathbb{C})$ as the non-primitive group $(P)$ is the only that fix smooth quartic surfaces. Moreover, these surfaces are in the pencil of quartics 
    $$\{\lambda_1(x_0^4)+\lambda_1(x_2 x_1^3+x_3^3 x_1+x_2^3 x_3)=0\},$$
    with singular fibers at the point $[1:0],[0:1]\in \mathbb{P}^1_{\lambda_0,\lambda_1}$. 
\end{theorem}

\begin{proof}
See Section \ref{QuarticInvariants}, Subsection  \ref{P}.
\end{proof}
\newpage
\appendix

\section{Algorithm: Invariant Forms of degree d}\label{Program}

%Describir con detalle ambos algoritmos, principalemente el discriminante.

{\scriptsize
\begin{algorithm}\caption{$G$-invariant forms of degree $d$}
\SetKwInput{KwInput}{Input}                % Set the Input
\SetKwInput{KwOutput}{Output}              % set the Output
\DontPrintSemicolon
  
  \KwInput{$n,d,m,(\sigma_i)_{1 \leq i \leq m},(A_i)_{1 \leq i \leq m}$ as appears in Section \ref{Invariant Subspaces}} 
  \KwOutput{Basis for each $\ker T_k,\, k \in K$}
  % \KwData{Testing set $x$}

% Set Function Names
  \SetKwFunction{FPolinomio}{GeneralPolynomial}
  \SetKwFunction{FLinearB}{LinearB}
  \SetKwFunction{FMatrixRepresentationB}{MatrixRepresentationB}
  \SetKwFunction{FMatrixRepresentationT}{MatrixRepresentationT}
  \SetKwFunction{FBasisT}{BasisT}
  
  \Begin{

  $C=\text{Array}(c,\delta)$\;
  $X=\text{Array}(x,n+1)$\;
  $E=\text{Exponents}(d,n+1)$\;
  $\delta = \text{Binomial}(n+d,n)$\;
  \SetKwProg{Fn}{Def}{:}{}
  \Fn{\FPolinomio{$Y, X$}}{
    \KwRet $\sum_{1 \leq i \leq \delta} Y_i \cdot \text{Monomial}(X,E_i)$\;
  }
  \SetKwProg{Fn}{Def}{:}{}
  \Fn{\FLinearB{$U,r,Y,X$}}{
    \KwRet $\text{GeneralPolyomial}(Y, U \cdot X) - r \cdot  \text{GeneralPolynomial}(Y,X)$\;
  } \;
  \SetKwProg{Fn}{Def}{:}{}
  \Fn{\FMatrixRepresentationB{$U,r$}}{
    %Consideremos $M,p,b$ como variables locales\; 
    $M = \text{ConstantArray}(0,\{\delta,\delta\})$\;
    $p = \text{LinearB}(U,r,A,X)$\;
    \For{$1 \leq i \leq \delta$}{
        $b=\text{Coefficient}(p,x_i)$\;
        \For{$1 \leq j \leq \delta$}{
            $M_{i,j} = \text{Coefficient}(b,c_j)$\;
        }
    }
    \KwRet $M$
  } \;
  \SetKwProg{Fn}{Def}{:}{}
  \Fn{\FMatrixRepresentationT{$U,\tau,m$}}{
    %Consideremos $P$ como variable local \;
    $P = \text{ConstantArray}(0,\{m \delta, \delta\})$\;
    \For{$1 \leq k \leq m$}{
        $M = \text{MatrixRepresentationB}(U_k,\tau_k)$\;
        \For{$1 \leq i \leq \delta$}{
            \For{$1 \leq j \leq \delta$}{
                $P_{i+(k-1)\delta,\, j} = M_{i,j}$
            }
        }
    }
    \KwRet $P$
  } \;
  \SetKwProg{Fn}{Def}{:}{}
  \Fn{\FBasisT{$U,\tau,m$}}{
    $P = \text{MatrixRepresentationT}(U,\tau,m)$\;
    $Q= \text{NullSpace}(P)$\;
    $q = \text{Lenght}(Q)$\;
    $R = \text{ConstantArray}(0,q)$\;
    \For{$1 \leq i \leq q$}{
        $R_i = \text{GeneralPolynomial}(Q_i,X)$\;
    }
    \KwRet $R$
  } \;
  Define $K$ as appears in Section \ref{Invariant Subspaces}. \;
  $l = \text{Length}(K)$\;
  $\text{G-InvariantForms} = \text{ConstantArray}(0,l)$\;
  \For{$1 \leq i \leq l$}{
    $\text{G-InvariantForms}_i = \text{BasisT}((A_j)_{1 \leq j \leq m}, (K_{i,j})_{1 \leq j \leq m},m)$\;
  }
  \KwRet \text{G-InvariantForms}
  
  }
\end{algorithm}
}

\pagebreak

\bibliographystyle{alpha}

\bibliography{main}

\end{document}